\documentclass{amsart}

\pdfoutput=1
\usepackage{amssymb}
\usepackage{amsthm}
\usepackage{amsmath}
\usepackage{graphicx}
\usepackage{subcaption}
\usepackage[color=white]{todonotes}
\usepackage{float}
\usepackage{enumitem}
\usepackage{color}

\title[Finite element convergence for the Joule heating problem]{Finite element convergence for the time-dependent Joule heating problem with mixed boundary conditions}
\author[Max Jensen]{Max Jensen\textsuperscript{1}}
\author[Axel M{\aa}lqvist]{Axel M{\aa}lqvist\textsuperscript{2}} 
\author[Anna Persson]{Anna Persson\textsuperscript{2}}

\newtheorem{definition}{Definition}[section]
\newtheorem{thm}[definition]{Theorem}
\newtheorem{lemma}[definition]{Lemma}
\newtheorem{corollary}[definition]{Corollary}
\theoremstyle{remark}
\newtheorem{remark}[definition]{Remark}

\numberwithin{equation}{section}

\DeclareMathOperator*{\esssup}{ess\,sup}

\DeclareMathOperator{\supp}{supp}

\DeclareMathOperator{\meas}{meas}

\newcommand{\R}{\mathbb{R}}
\newcommand{\N}{\mathbb{N}}
\newcommand{\sigmaup}{\sigma^\circ}
\newcommand{\sigmalow}{\sigma_\circ}
\newcommand{\intd}{\,\mathrm{d}}
\newcommand{\vv}{y}
\newcommand{\M}{\mathcal{M}}

\definecolor{dred}{RGB}{180,90,90}
\definecolor{dgreen}{RGB}{70,140,70}
\definecolor{dblue}{RGB}{100,100,180}

\usepackage{hyperref}

\begin{document}

\begin{abstract}
We prove strong convergence for a large class of finite element methods for the time-dependent Joule heating problem in three spatial dimensions with mixed boundary conditions on Lipschitz domains. We consider conforming subspaces for the spatial discretization and the backward Euler scheme for the temporal discretization. Furthermore, we prove uniqueness and higher regularity of the solution on creased domains and additional regularity in the interior of the domain. Due to a variational formulation with a cut-off functional the convergence analysis does not require a discrete maximum principle, permitting approximation spaces suitable for adaptive mesh refinement, responding to the the difference in regularity within the domain.	
\end{abstract}
\maketitle

\footnotetext[1]{Department of Mathematics, University of Sussex, Brighton BN1 9QH, United Kingdom.}
\footnotetext[2]{Department of Mathematical Sciences, Chalmers University of Technology and University of Gothenburg, SE-412 96 Gothenburg, Sweden.}

\section{Introduction}
The time-dependent Joule heating problem is a coupled non-linear elliptic-parabolic system of the form
\begin{equation}\label{JH}
\dot u - \Delta u = \sigma(u)|\nabla \varphi|^2,\quad
\nabla \cdot \sigma(u) \nabla \varphi = 0,
\end{equation}
where $u$ denotes the temperature and $\varphi$ the electric potential. It models the heat flow generated when an electric current is passed through a conductor. In applications the electric potential is typically only applied to smaller parts of the boundary, for instance through electric pads. To model such problems properly mixed boundary conditions are needed, see, e.g.~\cite{Henneken06}.

The Joule heating problem has been studied both in a theoretical context \cite{Cimatti92, Antontsev94, Yuan94, Meinlschmidt17}, focusing on the well-posedness of \eqref{JH}, and from a numerical point of view \cite{Elliott95, Akrivis05, Gao14, Li14}, focusing on convergence (with rate) of numerical solutions to \eqref{JH}. There are also several works on the stationary version of the problem, see, for instance \cite{Howison93, Holst10, Jensen13} and references therein.

The main issue with the system \eqref{JH} is the low regularity of the source term $\sigma(u)|\nabla \varphi|^2$. In one and two dimensions this does not lead to a problem. However, in three dimensions this term is not in $H^{-1}$ and the problem does not fit into the classical variational framework for PDEs. In \cite{Antontsev94} this issue is resolved by rewriting the source term using the equation for $\varphi$ (see also \cite{Howison93} for the stationary case). With this formulation existence of a solution in $L_2(H^1)$ is proved. However, to derive convergence for finite element approximations additional regularity of the solution is usually required, see \cite{Elliott95, Akrivis05}. Typically, sufficient regularity in three dimensions cannot be proved, but needs to be assumed. To the authors' knowledge, there is no numerical analysis of this problem under more realistic assumptions on the domain (Lipschitz in three spatial dimensions) and the boundary conditions (mixed).

The purpose of this paper is to prove the strong convergence of finite element approximations of \eqref{JH} on Lipschitz domains in three spatial dimensions with mixed boundary conditions. A challenge is to avoid the need for a discrete maximum principle and the associated restrictive mesh conditions \cite{Holst10} because a direct energy argument only delivers $L^1$-control on the critical $|\nabla \varphi|^2$ term in \eqref{JH}. In our analysis this is achieved by introducing a variational formulation with a cut-off functional, extending \cite{Jensen13}. The analysis presented in this paper covers finite element methods of any order that are conforming in space and piecewise constant in time, satisfying a backward Euler scheme. The choice of approximation spaces only needs to ensure the stability of the $L^2$ projection in the $H^1$-norm, which holds for a large class of non-uniform meshes \cite{Bank14}.

Having arrived at only mild mesh conditions, we find the Joule heating problem with mixed boundary conditions well suited for adaptive mesh refinement. Indeed, starting from the assumption of creased domains \cite{Mitrea07}, we prove uniqueness and additional regularity of the solution. This result combines the regularity for the Poisson equation on creased domains in \cite{Mitrea07} with the results for parabolic systems in \cite{Hieber08}. The additional regularity we obtain for $\varphi$, namely $\varphi \in L_{2q/(q-3)}(W^1_q)$ for some $q>3$, is in line with the sufficient condition for uniqueness established in~\cite{Antontsev94}. Importantly, we can show higher regularity, in some cases $C^\infty$, in the interior of the domain. To exploit the difference in regularity within the domain we equip the Joule heating problem with a goal functional to examine duality-based additive mesh refinement in the numerical experiments.

The paper is outlined as follows: In Section 2 we formulate the problem of interest and introduce some notation. Section 3 is devoted to the analysis of semi-discrete methods and Section 4 to fully discrete methods. In Section 5 we prove additional regularity and uniqueness of the solution. In Section 6 we present some numerical examples that confirm the convergence results and investigate adaptive mesh refinements.

\section{Variational formulations and weak solutions}
In this section we introduce two variational formulations, one ``classical'', see \eqref{classicalweak} below, and one based on a cut-off functional, see \eqref{weak} below. We prove that these two are equivalent, that is, that they have the same set of solutions. The latter formulation is preferable when working with finite element discretizations of the problem, since we avoid using a discrete maximum principle, see Section~\ref{sec_semidiscrete} and Section~\ref{sec_discrete}.

\subsection{Problem formulation and notation}
Let $D_t$ denote the time derivative $\frac{\partial }{\partial t}$ and $\Omega\subseteq \R^3$ be a domain describing the body of a conductor. Let $u\colon\Omega \times [0,T] \rightarrow \R$ denote the temperature inside the conductor, $\varphi\colon \Omega\times [0,T] \rightarrow \R$ the electric potential, and $\sigma\colon \R \rightarrow \R_{+}$ the electric conductivity. Furthermore, we use $\Gamma^u_D$ and $\Gamma^u_N$ to denote the Dirichlet and Neumann boundary for $u$ and $\overline{\Gamma^u_D} \cup \overline{\Gamma^u_N} = \partial \Omega$. Analogously, we define $\Gamma^\varphi_D$ and $\Gamma^\varphi_N$ for $\varphi$. With this notation, the time-dependent Joule heating problem is given by the following nonlinear elliptic-parabolic system
\begin{subequations}\label{jouleheating}
	\begin{alignat}{2}
		D_t u - \Delta u &= \sigma(u) |\nabla \varphi|^2,& \quad &\text{in } \Omega \times (0,T),\label{joule1}\\
		\nabla\cdot (\sigma(u)\nabla \varphi) &= 0,& &\text{in } \Omega \times (0,T),\label{joule2}\\
		u &= g_u,& &\text{on } \Gamma^u_{D} \times (0,T),\label{bc1}\\
		\varphi &= g_\varphi,& &\text{on } \Gamma^\varphi_{D} \times (0,T),\label{bc2}\\
		n \cdot \nabla u &= 0,& &\text{on } \Gamma^u_{N} \times (0,T),\label{bc3}\\
		n \cdot \nabla \varphi &= 0,& &\text{on } \Gamma^\varphi_{N} \times (0,T),\label{bc4}\\
		u(\cdot,0) &= u_0,& &\text{in } \label{initial}\Omega.
	\end{alignat}
\end{subequations} 

Let $W^k_p(\Omega)$ denote the classical range of Sobolev spaces and define
\begin{align*}
	W^k_p(\Omega;\Gamma^u_D) := \{v\in W^k_p(\Omega): v|_{\Gamma^u_D}=0\}, \quad \text{for } k>1/p.
\end{align*} 
The space $W^k_p(\Omega;\Gamma^\varphi_D)$ is defined analogously and $H^1$ is used to denote $W^1_2$. We also use $V^\ast$ for the dual space to $V$. Furthermore, we adopt the notation $L_p(0,T;V)$ for the Bochner space with norm
\begin{align*}
	\|v\|_{L_p(0,T;V)} &= \Big(\int_0^T \|v\|_V^p \, \mathrm{dt} \Big)^{1/p}, \quad 1\leq p<\infty,\\
	\|v\|_{L_\infty(0,T;V)} &= \esssup_{0\leq t \leq T} \|v\|_V,
\end{align*}
where $V$ is a Banach space equipped with the norm $\|\cdot\|_V$. The notation $v\in H^1(0,T;V)$ is used to denote $v,D_t v\in L_2(0,T;V)$. Finally, $C_b(\Omega)$ is the space of bounded continuous functions.

\subsection{Classical variational formulation}
To this end we make the following assumptions on the domain and the data.
\begin{enumerate}[label=(A\arabic*)]
	\item $\Omega\subseteq \R^3$ is a bounded domain with Lipschitz boundary, $\meas(\Gamma^u_D)>0$, and $\meas(\Gamma^\varphi_D)>0$. \label{ass_omega}
	\item  $g_u \in L_2(0,T;H^1(\Omega))\cap H^1(0,T;H^1(\Omega)^\ast)$ and there are points \\ $0 = t_0 < t_1 < \cdots < t_K = T$, $K \in \mathbb{N}$, such that
  \begin{align*}
    D_t g_u & \in C_b([t_i,t_{i+1});H^1(\Omega)^\ast),\\
    g_\varphi & \in C_b([t_i,t_{i+1}); W^1_3(\Omega)\cap L_\infty(\Omega)),
  \end{align*}
  on each subinterval $[t_i,t_{i+1})$. \label{ass_g}
	\item $u_0 \in L_2(\Omega)$. \label{ass_intial}
	\item $\sigma \in C^1(\R)$, Lipschitz continuous, and $0< \sigmalow \leq \sigma(x) \leq \sigmaup < \infty$, $\forall x \in \R$. \label{ass_sigma}
\end{enumerate}

A weak solution to the Joule heating problem \eqref{jouleheating} is a pair \\ $(u,\varphi)=(g_u + \tilde u, g_\varphi + \tilde \varphi)$ such that 
\begin{align*}
	(\tilde u, \tilde \varphi)\in L_2(0,T;H^1(\Omega;\Gamma^u_D)) \cap H^1(0,T;H^1(\Omega;\Gamma^u_D)^\ast) \times L_2(0,T;H^1(\Omega,\Gamma^\varphi_D))
\end{align*}
and for a.e.~$t \in (0,T]$
\begin{subequations}\label{classicalweak}
	\begin{align}
		\langle D_t u, v \rangle + \langle \nabla u, \nabla v \rangle &= \langle \sigma(u)|\nabla \varphi|^2, v \rangle ,\label{classicalweaku}\\
		\langle \sigma(u) \nabla \varphi, \nabla w \rangle &= 0, \label{classicalweakphi}\\
		\langle u(0), z \rangle &= \langle u_0, z \rangle,\label{classicalweak0}
	\end{align}
\end{subequations}
for all $(v,w)\in W^1_\infty(\Omega; \Gamma^u_D) \times H^1(\Omega;\Gamma^\varphi_D)$ and $z \in L_2(\Omega)$, see, for instance, \cite{Cimatti92}. Note that $\langle \cdot, \cdot \rangle$ is used to denote both the inner product in $L_2$ and the duality bracket. The choice of spaces guarantees $\sigma(u)|\nabla \varphi|^2 \in L_1(0,T;L_1(\Omega))$ so that the right-hand side in \eqref{classicalweaku} is well-defined for all $v \in W^1_\infty(\Omega; \Gamma^u_D)$.

Throughout the text we adopt the notational convention that for a function $\flat$ one understands $\tilde \flat = \flat - g_\varphi$ if $\flat$ is a Greek letter and $\tilde \flat = \flat - g_u$ if $\flat$ is a Latin letter.

\begin{remark}
	In some works, e.g.~\cite{Roubicek13}, the notion of strong (instead of weak) solution is used when the equation is satisfied almost everywhere in time. 
\end{remark}

The following lemma provides a maximum principle for $\varphi(x,t)$.

\begin{lemma}\label{phi_reg}
	If $(u,\varphi)$ is a solution to \eqref{classicalweak}, then 
	$g_\circ \leq \varphi(x,t) \leq g^\circ$ for a.e.~$(x,t)\in \bar\Omega\times[0,T]$, where
	\begin{align*}
		g^\circ := \max_{(x,t) \in \Gamma^\varphi_D \times [0,T]} g_\varphi(x,t), \quad g_\circ := \min_{(x,t) \in \Gamma^\varphi_D \times [0,T]} g_\varphi(x,t),
	\end{align*}
\end{lemma}
\begin{proof}
	Define $\chi = \max(0,\varphi-g^\circ) \in L_2(0,T;H^1(\Omega;\Gamma^\varphi_D))$ and choose $w=\chi(t)$ in \eqref{classicalweakphi} and integrate from $0$ to $T$. Then
	\begin{align*}
		0 &= \int_0^T \langle \sigma(u) \nabla \varphi, \nabla \chi \rangle \intd t = \int_0^T \langle \sigma(u) \nabla (\varphi-g^\circ), \nabla \chi \rangle \intd t \\&= \int_0^T \int_{\supp(\chi)\cap \Omega}\sigma(u) \nabla \chi \cdot \nabla \chi \intd x\intd t = \int_0^T \langle \sigma(u) \nabla \chi, \nabla \chi \rangle \intd t.
	\end{align*}
	Using $\sigma(u)\geq \sigma_\circ$ and the Poincar\'e-Friedrichs inequality we get $\int_0^T\|\chi\|^2\intd t=0$ and we deduce $\varphi \leq g^\circ$. A similar argument using $g_\circ$ proves $\varphi \geq g_\circ$. This gives $\varphi \in L_\infty(0,T;L_\infty(\Omega))$. 
\end{proof}

In one and two spatial dimensions the formulation \eqref{classicalweak} is suitable for proving existence of a solution, see, e.g.~\cite{Cimatti92, Elliott95}. However, because of the low regularity of the right-hand side in \eqref{classicalweaku} this strategy does not apply to the three dimensional setting. To overcome this difficulty, it can be proved that due to \eqref{joule2}, see, for instance, \cite{Antontsev94,Howison93},
\begin{align}\label{rewrite}
	\sigma(u) |\nabla \varphi|^2 = \nabla \cdot (\sigma(u)\varphi\nabla \varphi),
\end{align}
and from Lemma~\ref{phi_reg} it follows that $\nabla \cdot (\sigma(u)\varphi\nabla \varphi) \in L_2(0,T;H^1(\Omega,\Gamma^u_D)^\ast)$. With this right-hand side it is now possible to use Schauder's fixed point theorem to prove existence of a solution also in three dimensions.

\begin{thm}\label{existence_uniqueness_classical}
		There exists a solution $(u,\varphi)$ to \eqref{classicalweak}. If $\nabla \varphi \in L_{\frac{2q}{q-3}}(0,T;L_q(\Omega))$ for some $q>3$, then the solution is unique. 
\end{thm}
\begin{proof}
	This follows by adapting the fixed point argument in \cite[Theorem~2.2]{Antontsev94} to mixed boundary conditions. The proof uses identity \eqref{rewrite} and Schauder's fixed point theorem on the space $L_2(0,T;L_2(\Omega))$. More precisely, we consider the mapping $F:L_2(0,T;L_2(\Omega)) \rightarrow L_2(0,T;L_2(\Omega))$ where $y=F(s)$ is the solution to
	\begin{align*}
	\langle D_t y, v \rangle + \langle \nabla y, \nabla v \rangle &= \langle \nabla \cdot (\sigma(s)\psi\nabla \psi), v \rangle ,\\
	\langle \sigma(s) \nabla \psi, \nabla w \rangle &= 0\\
	\langle y(0), z \rangle &= \langle u_0, z \rangle,
	\end{align*}
	for all $(v,w)\in H^1(\Omega; \Gamma^u_D) \times H^1(\Omega;\Gamma^\varphi_D)$ and $z \in L_2(\Omega)$. It is clear, via \eqref{rewrite} and the fact that $W^1_\infty(\Omega; \Gamma^u_D)$ is dense in $H^1(\Omega; \Gamma^u_D)$, that a fixed point to $F$ solves \eqref{classicalweak}. To prove that $F$ satisfies the conditions of Schauder's fixed point theorem on some ball $B_R$ we may now follow \cite[Theorem~2.2]{Antontsev94}. The mixed boundary conditions only affect the definition of the space $V \subset H^1(\Omega)$, that is, the functions in $V$ in our case only vanish on $\Gamma_D \subset \partial \Omega$.
	
	The uniqueness follows from \cite[Theorem~4.1]{Antontsev94}, also by first adapting the argument to mixed boundary conditions.
\end{proof}

\subsection{Variational formulation with cut-off}
In this paper we are interested in proving convergence of finite element approximations. For this purpose, we propose a variational formulation based on a cut-off functional to avoid using a discrete maximum principle. The cut-off functional was introduced for the stationary problem in \cite{Jensen13}, and is defined as
\begin{align*}
\lceil f \rceil := \min\{\max\{f + g_\varphi, a \}, b\}-g_\varphi.
\end{align*}
for some fixed $a, b \in \mathbb{R}$ with $a \le g_\circ$ and $b \ge g^\circ$. Note that $\min$ and $\max$ are taken over both space and time $\Omega\times[0,T]$ and we have $a - g_\varphi \leq \lceil f \rceil \leq b - g_\varphi$.

To introduce the new weak formulation we define
\begin{align*}
	X &:= L_2(0,T;H^1(\Omega;\Gamma^u_D)) \cap H^1(0,T;H^1(\Omega;\Gamma^u_D)^\ast)\times L_2(0,T;H^1(\Omega,\Gamma^\varphi_D)), \\
	Y &:= H^1(\Omega;\Gamma^u_D) \times H^1(\Omega;\Gamma^\varphi_D).
\end{align*}
Using these spaces, a weak solution to the system \eqref{jouleheating} is a pair \\$(u,\varphi)=(g_u + \tilde u, g_\varphi + \tilde \varphi)$ such that $(\tilde u, \tilde \varphi)\in X$ and for a.e.~$t\in (0,T]$
\begin{subequations}\label{weak}
	\begin{align}
		\langle D_t u, v \rangle + \langle \nabla u, \nabla v \rangle &= -\langle \sigma(u)\lceil \tilde \varphi \rceil \nabla \varphi, \nabla v \rangle  +  \langle \sigma(u) \nabla \varphi \cdot \nabla g_\varphi, v\rangle,\label{weaku} \\
		\langle \sigma(u) \nabla \varphi, \nabla w \rangle &= 0, \label{weakphi}\\
		\langle u(0), z \rangle &= \langle u_0, z \rangle, \label{weakintial}
	\end{align}
\end{subequations}
for all $(v,w)\in Y$ and $z \in L_2(\Omega)$. 

\begin{lemma}\label{weak_forms_eq}
		The set of solutions which satisfy \eqref{weak} is equal to the set of solutions to \eqref{classicalweak}. In particular, the right-hand side in \eqref{weaku} defines an element in $L_2(0,T; (H^1(\Omega; \Gamma^u_D))^\ast)$. 
\end{lemma}
\begin{proof}
	The identity
	\begin{align*}
		\langle \sigma(u) |\nabla \varphi|^2, v \rangle = -\langle \sigma(u)\tilde \varphi\nabla \varphi,\nabla v \rangle + \langle \sigma(u) \nabla \varphi \cdot \nabla g_\varphi, v\rangle,
	\end{align*}
	for $v \in  W^1_\infty(\Omega)$ and a.e.~$t\in [0,T]$ follows by choosing $w = (\varphi(t)-g_\varphi(t))v$ in \eqref{classicalweakphi}, see \cite[Lemma~1]{Howison93} for the stationary case. The definition of the cut-off functional and the maximum principle for $\varphi$ in Lemma~\ref{phi_reg} implies that $\lceil \tilde \varphi \rceil = \tilde \varphi$. The larger space of test functions does not affect the set of solutions since $W^1_\infty(\Omega;\Gamma^u_D)$ is dense in $H^1(\Omega;\Gamma^u_D)$.
	
	Furthermore, the right-hand side in \eqref{weaku} satisfies the following bound
	\begin{align*}
		|-\langle \sigma(u)\lceil \tilde \varphi \rceil \nabla \varphi,& \nabla v \rangle +  \langle \sigma(u) \nabla \varphi \cdot \nabla g_\varphi, v\rangle| \leq C(\sigma, g_\varphi)  \|\nabla \varphi\|_{L_2(\Omega)} \|\nabla v\|_{L_2(\Omega)} \\&\quad+ \sigma^\circ \|\nabla \varphi\|_{L_2(\Omega)}\|\nabla g_\varphi\|_{L_3(\Omega)} \|v\|_{L_6(\Omega)},
	\end{align*}
	where the Sobolev embedding in $\R^3$ gives $\|v\|_{L_6(\Omega)} \leq C\|v\|_{H^1(\Omega)}$. Hence 
	\begin{align*}
	\int_0^T &\|\nabla \cdot (\sigma(u)\lceil \tilde \varphi \rceil \nabla \varphi) + \sigma(u) \nabla \varphi \cdot \nabla g_\varphi\|^2_{H^1(\Omega;\Gamma^u_D)^\ast} \intd t \\&\leq C(\sigma,g_\varphi)(\|\nabla \varphi\|_{L_2(0,T;L_2(\Omega))} + \|\nabla \varphi\|_{L_2(0,T;L_2(\Omega))}\|\nabla g_\varphi\|_{L_\infty(0,T;L_3(\Omega))}),
	\end{align*}
	where $\|\nabla g_\varphi\|_{L_\infty(0,T;L_3(\Omega))}$ is bounded due to assumption \ref{ass_g}, so the right-hand side defines an element in $L_2(0,T;H^1(\Omega; \Gamma^u_D)^\ast)$.
\end{proof}


\section{Semidiscrete methods}\label{sec_semidiscrete}
In this section we analyze spatially semidiscrete Galerkin methods. We prove existence and uniqueness of semidiscrete solutions and strong convergence to a weak solution satisfying \eqref{weak}.

\subsection{Semidiscrete formulation}
Let $\{V^u_m\}_{m \in \N}$ and $\{V^\varphi_m\}_{m \in \N}$  be hierarchical families of finite-dimensional subspaces, whose unions are dense in $H^1(\Omega; \Gamma^u_D)$ and $H^1(\Omega; \Gamma^\varphi_D)$, respectively and define
\begin{align*}
	X_m &:= \{ v \in C(0,T;V^u_m) : v|_{[t_i, t_{i+1})} \in C^1(t_i, t_{i+1};V^u_m) \; \forall \, i \} \times L_\infty(0,T;V^\varphi_m).
\end{align*}
Typically, $V^u_m$ and $V^\varphi_m$ are finite element spaces corresponding to a family of meshes $\{\mathcal T_m\}_{m \in \N}$. For instance, one may choose Lagrangian finite elements or conforming $hp$-finite elements, see also the numerical examples in Section~\ref{sec:examples}. 
	
We make the following additional assumption on $V^u_m$;
\begin{enumerate}[label=(A\arabic*)] \setcounter{enumi}{4}
	\item Let $V^u_m$ be of a form such that the $L_2$-projection $P_m$ onto $V^u_m$ is stable, uniformly in $m$, in the $H^1$-norm. \label{ass_proj}
\end{enumerate}
In the case when $V^u_m$ is a finite element space, we refer to \cite{Bank14} and references therein, where the $H^1$-stability of the $L_2$-projection is proved for a large class of (non-uniform) meshes in three spatial dimensions.

In the subsequent sections we let $C_m$ denote a generic constant that depends on the discretization $m$, for instance, the mesh size $h_m$.

A semidiscrete Galerkin solution is a pair $(u_m,\varphi_m)=(g_u + \tilde u_m, g_\varphi + \tilde \varphi_m)$ such that $(\tilde u_m, \tilde \varphi_m)\in X_m$ and for a.e.~$t \in (0,T]$
\begin{subequations}\label{semi_weak}
	\begin{align}
	\langle D_t u_m, v \rangle + \langle \nabla u_m, \nabla v \rangle &= -\langle \sigma(u_m)\lceil \tilde \varphi_m \rceil \nabla \varphi_m, \nabla v \rangle \label{semi_weak_u}\\&\quad+ \langle \sigma(u_m) \nabla \varphi_m \cdot \nabla g_\varphi, v\rangle, \notag\\
	\langle \sigma(u_m) \nabla \varphi_m, \nabla w \rangle &= 0, \label{semi_weak_phi} \\
	\langle u_m(0), z \rangle &= \langle u_0, z \rangle, \label{semi_weak_initial}
	\end{align}
\end{subequations}
for all $(v,w) \in V^u_m \times V^\varphi_m$ and $z \in V^u_m$. 

\begin{remark}\label{no_disc_max}
Recall the bound  of the cut-off functional; $a - g_\varphi \leq \lceil \tilde \varphi_m \rceil \leq b- g_\varphi$. This uniform boundedness of $\lceil \tilde \varphi_m \rceil$ in $m$ will allow us to consider the limit of $\langle \sigma(u_m)\lceil \tilde \varphi_m \rceil \nabla \varphi_m, \nabla v \rangle$ as $m \to \infty$ without appealing to a discrete maximum principle.
\end{remark}

\begin{lemma}\label{semi_bounds}
	A solution to \eqref{semi_weak} fulfils the following bounds
	\begin{align}
	\|\nabla \tilde \varphi_m\|_{L_\infty(0,T;L_2(\Omega))} &\leq C(\sigma,g_\varphi),\label{semi_bound_phi}\\
	\|\tilde u_m(T)\|^2_{L_2(\Omega)} + \int_0^T \|\nabla \tilde u_m\|^2_{L_2(\Omega)} \intd t &\leq  C(u_0,\sigma,g_u,D_t g_u, g_\varphi),\label{semi_bound_u}\\
	\int_0^T\|D_t \tilde u_m(t)\|^2_{H^1(\Omega;\Gamma^u_D)^\ast} &\leq  C(u_0,\sigma,g_u,D_t g_u, g_\varphi).\label{semi_bound_u_t}
	\end{align}
\end{lemma}
\begin{proof}
	By choosing $w=\tilde \varphi_m(t)$ in \eqref{semi_weak_phi} we can prove 
	\begin{align*}
	\|\nabla \tilde \varphi_m(t)\|^2_{L_2(\Omega)} \leq C(\sigma)\|\nabla g_\varphi(t)\|^2_{L_2(\Omega)},
	\end{align*}
	and \eqref{semi_bound_phi} follows by using \ref{ass_g} and \ref{ass_sigma}. 
	
	By choosing $v=\tilde u_m(t)$ in \eqref{semi_weak_u} and integrating from $0$ to $T$ we have
	\begin{align*}
	& \int_0^T \langle D_t \tilde u_m, \tilde u_m \rangle \intd t +\int_0^T \| \nabla \tilde u_m \|^2_{L_2(\Omega)} \intd t\\
= &\,  -\int_0^T \langle D_t g_u, \tilde u_m \rangle \intd t - \int_0^T \langle \nabla g_u, \nabla \tilde u_m \rangle \intd t - \int_0^T \langle \sigma(u_m)\lceil\tilde \varphi_m\rceil \nabla \varphi_m, \nabla\tilde u_m \rangle \intd t\\
& \, + \int_0^T \langle \sigma(u_m)\nabla \varphi_m\cdot \nabla g_\varphi, \tilde u_m \rangle \intd t \\ 
= & \; I + II + III + IV.
	\end{align*}
	Using the Cauchy-Schwarz, Poincar\'{e}, and Young's (weighted) inequality we get
	\begin{align} \label{semi_est_1}
	I + II \leq \frac{1}{4} \int_0^T \|\nabla \tilde u_m\|^2_{L_2(\Omega)} \intd t + C\int_0^T\|D_t g_u\|^2_{(H^1(\Omega;\Gamma^u_D))^\ast} + \|\nabla g_u\|^2_{L_2(\Omega)} \intd t,
	\end{align} 
	and
	\begin{align}\label{semi_est_2}
	III + IV &\leq
	\frac{1}{4}\int_0^T \|\nabla \tilde u_m\|^2_{L_2(\Omega)}\intd t + C(\sigma,g_\varphi)\Big(\|g_\varphi\|^2_{L_\infty(0,T;W^1_3(\Omega))} \\&\quad + \int_0^T\|\nabla \varphi_m\|^2_{L_2(\Omega)} \intd t \Big), \notag
	\end{align}
	where we used Sobolev embeddings as in the proof of Lemma~\ref{weak_forms_eq}. We can now use \eqref{semi_bound_phi} to bound the last term on the right-hand side. Finally, using \eqref{semi_weak_initial} we have
	\begin{align*}
	2\int_0^T \langle D_t{\tilde u}_m, \tilde u_m \rangle \intd t &= \int_0^T D_t \|\tilde u_m\|^2_{L_2(\Omega)} = \|\tilde u_m(T)\|^2_{L_2(\Omega)} - \|\tilde u_m(0)\|^2_{L_2(\Omega)} \\& \geq \|\tilde u_m(T)\|^2_{L_2(\Omega)} - \|u_0\|^2_{L_2(\Omega)} - \|g_u(0)\|^2_{L_2(\Omega)},
	\end{align*}
	and \eqref{semi_bound_u} follows.

	Observe that $D_t \tilde u_m$ belongs to $V^u_m$. With $P_m$ denoting the $L_2$-projection onto $V^u_m$ we get 
	\begin{align*} 
	\|D_t \tilde u_m(t)\|_{H^1(\Omega; \Gamma^u_D)^\ast} &= \sup_{v \in H^1(\Omega; \Gamma^u_D) \atop v \neq 0} \frac{\langle D_t \tilde u_m(t), v \rangle}{\|v\|_{H^1(\Omega)}} = \sup_{v \in H^1(\Omega; \Gamma^u_D) \atop v \neq 0} \frac{\langle D_t \tilde u_m(t), P_m v \rangle}{\|v\|_{H^1(\Omega)}}\\
&\leq \sup_{v \in H^1(\Omega; \Gamma^u_D)\atop P_m v \neq 0} C\frac{\langle D_t \tilde u_m(t), P_m v \rangle}{\|P_m v\|_{H^1(\Omega)}}\notag = \sup_{v \in V^u_m \atop v \neq 0} C\frac{\langle D_t \tilde u_m(t), v \rangle}{\|v\|_{H^1(\Omega)}},
	\end{align*}  
where $C$ is the $H^1$-norm of the $L_2$-projection, which is independent of $m$ due to \ref{ass_proj}.
Now we can use bounds similar to \eqref{semi_est_1} and \eqref{semi_est_2} to prove \eqref{semi_bound_u_t}.
\end{proof}

\begin{lemma}\label{semi_existence_uniqueness} 
	There exists a unique solution $(\tilde u_m,\tilde \varphi_m)\in X_m$ to \eqref{semi_weak}. 
\end{lemma}
\begin{proof}

For each $\tilde u_m$ there is a solution $\tilde \varphi_m = S(\tilde u_m)$ of \eqref{semi_weak_phi}, which defines a mapping $S:L_2(0,T;V^u_m)\rightarrow L_2(0,T;V^\varphi_m)$.

First, we assume that $g_\varphi(x,\cdot)$ and $D_t g_u(x,\cdot)$ are continuous to prove existence and uniqueness on $[0,T]$. Let $\{\lambda_i\}_{i=1}^M$ be a basis for $V^u_m$. Then $\tilde u_m = \sum_{j=1}^M \alpha_j(t) \lambda_j$ for some $\alpha(t) = (\alpha_j(t)) \in \R^M$. By substituting into \eqref{semi_weak_u} we arrive at the following system of ODEs
	\begin{align}\label{ODE_alpha}
		M D_t \alpha(t) + K \alpha(t) = F(\alpha(t),t) + G(t),
	\end{align}
	where $M$ and $K$ denote the mass and stiffness matrices, respectively, and
	\begin{align*}
	F_i(\alpha(t),t) &= -\langle \sigma(u_m(t))\lceil \tilde \varphi_m(t) \rceil \nabla \varphi_m(t), \nabla \lambda_i \rangle + \langle \sigma(u_m(t)) \nabla \varphi_m(t) \cdot \nabla g_\varphi(t), \lambda_i\rangle, \\
	G_i(t) &= -\langle D_t g_u(t), \lambda_i \rangle - \langle \nabla g_u(t), \nabla \lambda_i \rangle,
	\end{align*}
where $\tilde \varphi_m = S(\tilde u_m)$. The initial data is given by \eqref{semi_weak_initial} and corresponds to the equation $M \alpha(0) = b$, where $b_i = \langle u_0-g_u(0), \lambda_i \rangle$. 
	
	Let $\tilde u^1_m = \sum_{j=1}^M \alpha^1_j\lambda_j$ and $\tilde u^2_m = \sum_{j=1}^M \alpha^2_j\lambda_j$, and $\tilde \varphi^1_m := S(\tilde u^1_m)$ and $\tilde \varphi^2_m := S(\tilde u^2_m)$. Note that
	\begin{align*}
		\|\nabla(\tilde u^1_m &- \tilde u^2_m)\|_{L_2(\Omega)} \leq C_m \, \|\alpha^1-\alpha^2\|_{\R^M},
	\end{align*}
due to the Lipschitz continuity of linear mappings in $V^u_m$.
	For $\tilde \varphi^1_m- \tilde \varphi^2_m$ we use Strang's first Lemma \cite[Chapter III]{Braess07} and the Lipschitz continuity of $\sigma$ to get
	\begin{align*}
		\sigma_\circ \|\nabla(\varphi^1_m-\varphi^2_m)\|_{L_2(\Omega)} &\leq \|(\sigma(u^1_m)- \sigma(u^2_m))\nabla  \varphi^1_m\|_{L_2(\Omega)} \\&\leq C\|\tilde u^1_m- \tilde u^2_m\|_{L_2(\Omega)}\|\nabla  \varphi^1_m\|_{L_\infty(\Omega)} \leq C_m\|\alpha^1-\alpha^2\|_{\R^M},
	\end{align*}
	which means that the mapping $S$ is Lipschitz continuous. Here we have used that the boundary data is identical for the two instances, that is, $u^1_m-u^2_m = \tilde u^1_m-\tilde u^2_m$ and $\varphi^1_m-\varphi^2_m=\tilde \varphi^1_m-\tilde \varphi^2_m$.
	
	Now, for $F_i$ we have for $i=1,...,M$
	\begin{align*}
		|F_i(\alpha^1(t),t) - F_i(\alpha^2(t),t)| &\leq |\langle \sigma(u^1_m)\lceil \tilde \varphi^1_m \rceil \nabla \varphi^1_m - \sigma(u^2_m)\lceil \tilde \varphi^2_m \rceil \nabla \varphi^2_m, \nabla \lambda_i \rangle| \\&\quad+ |\langle (\sigma(u^1_m) \nabla \varphi^1_m -\sigma(u^2_m) \nabla \varphi^2_m) \cdot \nabla g_\varphi, \lambda_i\rangle|.
	\end{align*}
Note that $\sigma$ and the cut-off functional $\lceil \cdot \rceil$ are Lipschitz continuous and bounded. Furthermore, the image of $S(\cdot)$ is bounded owing to Lemma~\ref{semi_bounds}. Using this, together with the fact that the product of bounded Lipschitz continuous functions is Lipschitz continuous, we prove the Lipschitz continuity of $F$ in $t$;
	\begin{align*}
		\|F(\alpha^1(t),t) - F(\alpha^2(t),t)\|_{\R^M} \leq C_m\|\alpha^1(t)-\alpha^2(t)\|_{\R^M},
	\end{align*}
	where $C_m$ is a generic constant that does not depend on $t$. 
	
Picard-Lindel\"{o}f's theorem gives existence and uniqueness on some maximal interval $(\beta_1,\beta_2)$. If $(\beta_1,\beta_2)$ is a strict subset $(0,T)$ then by \cite[Theorem~7.6]{Amann90} either
	\begin{align*}
		\lim_{t \rightarrow \beta_1^{+}} \|\alpha(t)\|_{\R^M} = \infty, \qquad \text{or} \qquad \lim_{t \rightarrow \beta_2^{-}} \|\alpha(t)\|_{\R^M} = \infty,
	\end{align*}  
which contradicts Lemma~\ref{semi_bounds}. Hence there exist a unique solution in $C^1(0,T;V^u_m)\times L_\infty(0,T;V^\varphi_m)$.
	
	Finally, we consider the case when $D_t g_u(x,\cdot)$ and $g_\varphi(x,\cdot)$ have at most finitely many discontinuities as specified in \ref{ass_g}. We may then use Picard-Lindel\"{o}f's theorem on the sub-interval $[t_{k},t_{k+1})$ with the initial data $\alpha(t_{k}) = \lim_{t\rightarrow t_{k}^{-}} \alpha(t)$. The existence and uniqueness on $[0,T]$ now follows by induction over $k$.
\end{proof}

\subsection{Convergence of semidiscrete solutions}
The following lemma will be used several times in the convergence analysis in the subsequent text. Recall that $\tilde \flat = \flat - g_\varphi$ if $\flat$ is a Greek letter and $\tilde \flat = \flat - g_u$ if $\flat$ is a Latin letter.

\begin{lemma} \label{lem_second_equation}
	Consider a sequence $\{ \tilde \vv_m \}_m$ which converges pointwise a.e.~to $\tilde \vv$ and a sequence $\{\tilde \psi_m \}_m$ which converges weakly in $L_2(0,T;H^1(\Omega;\Gamma^\varphi_D))$ to $\tilde \psi$, and the corresponding sequences $\{ \vv_m \}_m$ and $\{\psi_m \}_m$ that converges to $\vv$ and $\psi$, respectively. Suppose that, for all $m \in \mathbb N$,
	\begin{align} \label{second_equation}
	\int_0^T \langle \sigma(\vv_m) \nabla \psi_m, \nabla \tilde \psi \rangle \intd t = \int_0^T \langle \sigma(\vv_m) \nabla \psi_m, \nabla \tilde \psi_m \rangle \intd t = 0. 
	\end{align}
	Then $\tilde \psi_m \to \tilde \psi$ strongly in $L_2(0,T;H^1(\Omega;\Gamma^\varphi_D))$ as $m \to \infty$. Furthermore, subsequences of \[\sigma(\vv_m) \nabla \psi_m \quad \text{and} \quad \sigma(\vv_m)\lceil\tilde \psi_m\rceil\nabla \psi_m\] converge strongly in $L_2(0,T;L_2(\Omega; \mathbb{R}^3))$ to $\sigma(\vv) \nabla \psi$ and $\sigma(\vv)\lceil\tilde \psi\rceil\nabla \psi$, respectively.
\end{lemma}

\begin{proof}
	The dominated convergence theorem implies that $\sigma(\vv_{m})\nabla \tilde \psi \rightarrow \sigma(\vv)\nabla \tilde \psi$ converges strongly in $L_2(0,T;L_2(\Omega; \mathbb{R}^3)) \simeq L_2( (0,T) \times \Omega; \mathbb{R}^3)$, as $|\sigma(\vv_{m}) \, \partial_i \psi| \le \sigma^\circ \| \nabla \psi \|$ for all $i, m$.
	
	Because in $L_2(0,T;L_2(\Omega))$ the scalar product of a bounded and weakly convergent sequence and a strongly convergent sequence converges to the scalar product of the limits,
	\begin{align} \label{conv_result_phi2}
	0 = & \int_0^T\langle\sigma(\vv_{m}) \nabla \psi_{m}, \nabla \tilde \psi \rangle \intd t = \int_0^T\langle \nabla \psi_{m}, \sigma(\vv_{m}) \nabla \tilde \psi \rangle \intd t\\
	\to &  \int_0^T\langle \sigma(\vv) \nabla \psi, \nabla \tilde \psi \rangle \intd t, \quad \text{as} \quad m \to \infty. \notag
	\end{align}
	In particular, $\int_0^T\langle \sigma(\vv) \nabla \psi, \nabla \tilde \psi \rangle \intd t = 0$. To prove strong convergence of $\tilde \psi_{m}$ we write
	\begin{align*}
	0 &\leq \int_0^T \langle \sigma(\vv_{m}) \nabla (\tilde \psi - \tilde \psi_{m}), \nabla (\tilde \psi - \tilde \psi_{m}) \rangle \intd t \\&= \int_0^T \big(\langle \sigma(\vv_{m}) \nabla \tilde \psi , \nabla \tilde \psi \rangle - 2\langle \sigma(\vv_{m}) \nabla \tilde \psi_{m}, \nabla \tilde \psi \rangle + \langle \sigma(\vv_{m}) \nabla \tilde \psi_{m}, \nabla \tilde \psi_{m} \rangle\big) \intd t\\
	&=: I + II + III.
	\end{align*}
	Using the strong convergence of $\sigma(\vv_{m})\nabla \tilde \psi$ we get
	\begin{align*}
	I \rightarrow \int_0^T \langle \sigma(\vv) \nabla \tilde \psi , \nabla \tilde \psi \rangle \intd t = - \int_0^T \langle \sigma(\vv) \nabla g_\varphi , \nabla \tilde \psi \rangle \intd t,
	\end{align*}
	and due to \eqref{conv_result_phi2} we have
	\begin{align*}
	II \rightarrow -2\int_0^T \langle \sigma(\vv) \nabla \tilde\psi , \nabla\tilde\psi \rangle \intd t = 2 \int_0^T \langle \sigma(\vv) \nabla g_\varphi , \nabla\tilde\psi \rangle \intd t.
	\end{align*}
	Now, due to \eqref{second_equation} the third term gives
	\begin{align*}
	III = -\int_0^T \langle \sigma(\vv_{m}) \nabla g_\varphi , \nabla \tilde\psi_m \rangle \intd t \rightarrow -\int_0^T \langle \sigma(\vv) \nabla g_\varphi , \nabla \tilde\psi \rangle \intd t,
	\end{align*}
	since, due to the dominated convergence theorem, $\sigma(\vv_{m}) \nabla g_\varphi \to \sigma(\vv) \nabla g_\varphi$ strongly in $L_2(0,T;L_2(\Omega))$ and $\nabla \tilde \psi_m$ converges weakly. Thus, we conclude that \\ $I+II+III \rightarrow 0$ and, by applying a Poincar\'{e}-Friedrichs inequality, that $\tilde \psi_{m}$ converges strongly in $L_2(0,T;H^1(\Omega;\Gamma^\varphi_D))$, and, by passing to a subsequence, also pointwise a.e. 
	
	Finally, by the dominated convergence theorem in the form of \cite[p.~270]{Royden88}, $\sigma(\vv_{m_k}) \nabla \psi_{m_k}$ and $\sigma(\vv_{m_k})\lceil\tilde \psi_{m_k}\rceil\nabla \psi_{m_k}$ converge strongly in $L_2(0,T;L_2(\Omega; \mathbb{R}^3))$,\\ where $m_k$ denotes a subsequence.
\end{proof}

\begin{thm}\label{semi_conv_galerkin}
A subsequence of solutions $(\tilde u_{m_k},\tilde \varphi_{m_k})\in X_{m_k}$ of \eqref{semi_weak} converges strongly in $X$ to a solution $(\tilde u, \tilde \varphi)$ of \eqref{weak}.
\end{thm}

\begin{proof}
Owing to Lemma~\ref{semi_bounds} and the reflexivity of $X$, there exists a subsequence $m_k$ and $(\tilde \vv, \tilde \psi) \in X$ such that
	\begin{align*}
	(\tilde u_{m_k}, \tilde \varphi_{m_k}) \rightharpoonup (\tilde \vv, \tilde \psi) \quad \text{in } X, \text{ as $k\rightarrow \infty$}.
	\end{align*} 
Because initial conditions are $L_2$-projected onto $V_m^u$ it follows that $\tilde u_{m_k}(0) \to \tilde \vv(0) = \tilde u(0)$ in $L_2(\Omega)$. 

The compactness of the embedding (Aubin-Lions lemma) $$L_2(0,T;H^1(\Omega;\Gamma^u_D))\cap H^1(0,T;H^1(\Omega;\Gamma^u_D)^\ast) \hookrightarrow L_2(0,T;L_2(\Omega)),$$ implies that there exist a subsequence, still denoted $m_k$, such that $\tilde u_{m_k}\rightarrow \tilde \vv$ strongly and pointwise a.e.~in $L_2(0,T;L_2(\Omega))$.

Owing to Lemma \ref{lem_second_equation} we can pass to subsequences, without change of notation, so that $\varphi_{m_k}$, $\nabla \varphi_{m_k}$, $\sigma(u_{m_k}) \nabla \varphi_{m_k}$ and $\sigma(u_{m_k})\lceil\tilde \varphi_{m_k}\rceil\nabla \varphi_{m_k}$ converge strongly in~$L_2(0,T,L_2(\Omega))$. Now choose $v(t)\in L_2(0,T;V^u_{m_k})$ in \eqref{semi_weak_u}. Integrating from $0$ to $T$ gives, 
\begin{align*}
\int_0^T\langle D_t u_{m_k}, v \rangle + \langle \nabla u_{m_k}, \nabla  v \rangle \intd t&= \int_0^T\big(-\langle \sigma(u_{m_k})\lceil \tilde \varphi_{m_k} \rceil \nabla \varphi_{m_k}, \nabla v \rangle \\&\quad+ \langle \sigma(u_{m_k}) \nabla \varphi_{m_k} \cdot \nabla g_\varphi,  v\rangle\big) \intd t. \notag
\end{align*}
Fixing $v$ we may now let $k \to \infty$ to get
that
\begin{align*}
\int_0^T\langle D_t \vv,  v \rangle + \langle \nabla \vv, \nabla  v \rangle \intd t&= \int_0^T-\langle \sigma(\vv)\lceil \tilde \psi \rceil \nabla \psi, \nabla  v \rangle + \langle \sigma(\vv) \nabla \psi \cdot \nabla g_\varphi,  v\rangle\intd t, \notag
\end{align*}
which holds for all $v \in L_2(0,T;H^1(\Omega;\Gamma^u_D))$, since $\cup_{k\in \N} L_2(0,T;V^u_{m_k})$ is dense in this space and $u_{m,k},\varphi_{m,k}$ are bounded according to Lemma~\ref{semi_bounds}, see \cite[Theorem~3,~p.~121]{Yosida95}. In the spirit of Lemma~\ref{lem_second_equation} we may also prove that $\tilde \psi$ satisfies \eqref{weakphi}. This together with the convergence of initial conditions imply that the limit $(\tilde \vv, \tilde \psi)$ is a solution to~\eqref{weak}.
	
	To prove that $\{ \tilde u_{m_k} \}_k$ converges strongly in $L_2(0,T;H^1(\Omega;\Gamma^u_D))$, we write
	\begin{align*}
		&\int_0^T\langle \nabla(\tilde \vv-\tilde u_{m_k}), \nabla(\tilde \vv-\tilde u_{m_k}) \rangle \intd t \\&\quad= \int_0^T \big(\langle \nabla\tilde \vv, \nabla\tilde \vv \rangle -2\langle \nabla\tilde \vv, \nabla\tilde u_{m_k} \rangle +\langle \nabla\tilde u_{m_k}, \nabla\tilde u_{m_k}\rangle\big)  \intd t =: I + II + III.
	\end{align*}
Then $II \rightarrow -2\int_0^T \langle \nabla\tilde \vv, \nabla\tilde \vv \rangle \intd t$ since $\tilde u_{m_k}$ converges weakly in $L_2(0,T;H^1(\Omega;\Gamma^u_D))$. For the third term we use \eqref{semi_weak_u} to get
	\begin{align*}
		III &= \int_0^T \big(\langle D_t \tilde u_{m_k}, \tilde u_{m_k} \rangle - \langle D_t g_u, \tilde u_{m_k} \rangle  -  \langle \nabla g_u, \nabla \tilde u_{m_k} \rangle \\&\quad- \langle \sigma(u_{m_k})\lceil\tilde \varphi_{m_k}\rceil \nabla \varphi_{m_k}, \nabla\tilde u_{m_k} \rangle  +  \langle \sigma(u_{m_k})\nabla \varphi_{m_k}\cdot \nabla g_\varphi, \tilde u_{m_k} \rangle \big) \intd t\\
		&\to \int_0^T \big(\langle D_t \tilde \vv, \tilde \vv \rangle - \langle D_t g_u, \tilde \vv \rangle  -  \langle \nabla g_u, \nabla \tilde \vv \rangle \\&\quad- \langle \sigma(\vv)\lceil\tilde \psi\rceil \nabla \psi, \nabla \tilde \vv \rangle  +  \langle \sigma(\vv)\nabla \psi \cdot \nabla g_\varphi, \tilde \vv\rangle \big) \intd t
	\end{align*}
as $k \to \infty$, recalling that $D_t \tilde u_{m_k}$ converges weakly, $u_{m_k}$ strongly, and the statement of Lemma \ref{lem_second_equation}. Now, \eqref{weaku} gives $III = \int_0^T  \langle \nabla\tilde \vv, \nabla\tilde \vv \rangle \intd t$ so that $I+II+III \rightarrow 0$.

Finally, to show strong convergence of the time derivative, we note that
\begin{align*}
\| D_t \tilde \vv - & D_t \tilde{u}_{m_k} \|_{H^1(\Omega; \Gamma_D^u)^\ast} \\& \leq \sup_{v \in H^1(\Omega;\Gamma^u_D) \atop v \neq 0} \frac{\langle D_t \tilde \vv - P_{m_k}D_t \tilde\vv, v \rangle}{\| v\|_{H^1(\Omega)}} + \sup_{v \in H^1(\Omega;\Gamma^u_D) \atop v \neq 0} \frac{\langle P_{m_k}D_t \tilde \vv - D_t \tilde{u}_{m_k}, v \rangle}{\| v\|_{H^1(\Omega)}} \\&=: a_{m_k} + b_{m_k},
\end{align*}
where $P_{m_k}$ is the $L_2$-projection onto $V^u_m$. It follows, due to the density of $L_2$ in~$(H^1)^*$, that $\int_0^T a^2_{m_k} \intd t \to 0$ as $k \to \infty$. Now, we use the self-adjointness of the $L_2$-projection, \eqref{weaku}, and \eqref{semi_weak_u} to get $\int_0^Tb_{m_k}^2\intd t\to 0$ as $k \to 0$ since $\nabla u_{m_k}$, $\sigma(u_{m_k})\lceil \tilde \varphi_{m_k} \rceil \nabla \varphi_{m_k}$, and $\sigma(u_{m_k})\nabla \varphi_{m_k}\cdot \nabla g_\varphi$ converge strongly in $L_2(0,T;L_2(\Omega))$. We conclude that $\| D_t \tilde \vv - D_t \tilde{u}_{m_k} \|_{L_2(0,T;H^1(\Omega; \Gamma_D^u)^\ast)} \to 0.$

\end{proof}

\begin{corollary}\label{cor_semi_conv}
If the solution $(u,\varphi)$ to \eqref{weak} is unique, then the full sequence of Galerkin solutions $(u_m,\varphi_m)\in X_m$ converges.
\end{corollary}

\begin{proof}
	Due to Lemma~\ref{semi_bounds} the sequence is bounded in $X$. From the proof of Theorem~\ref{semi_conv_galerkin} we deduce that any accumulation point of the sequence is a solution to \eqref{weak} and that an accumulation point exists. If the solution to \eqref{weak} is unique there can only be one accumulation point and, hence, the full sequence $\{(u_m,\varphi_m)\}_m$ must converge.
\end{proof}

\section{Fully discrete methods}\label{sec_discrete}
In this section we analyze fully discrete methods based on a backward Euler scheme in time and hierarchical families of finite dimensional subspaces, as introduced in Section~\ref{sec_semidiscrete}, in space. We prove existence and uniqueness of fully discrete solutions and strong convergence to a weak solution satisfying \eqref{weak}.

\subsection{Fully discrete formulation}
Let $\{J_l\}_{l\in \N }$ be a family of nested partitions of the time interval $J=[0,T]$, which subordinate to the decomposition of \ref{ass_g}. For each partition $0 = t_0 < t_1 < ... < t_N = T$ we denote the subintervals $I_n:=(t_{n-1},t_n]$ and $f^n:=f(t_n)$. We consider a uniform time discretization in the analysis, that is, we assume $t_n-t_{n-1}=\tau_l$ with $\tau_l=2^{-l}T$. It simplifies some of the analysis, but it is also a requirement for the compactness argument in \cite{Walkington10}.

Fix $m$ and let $V^u_m$ and $V^\varphi_m$ be as in Section~\ref{sec_semidiscrete} and define the discrete space
\begin{align*}
	X_{m,l} = &\{v(x,t): \forall n \ \exists w \in V^u_m: v(t,\cdot) = w, \ t\in I_n \} \\&\times \{v(x,t): \forall n \ \exists w \in V^\varphi_m: v(t,\cdot) = w, \ t\in I_n\}.
\end{align*}
This means that functions in $X_{m,l}$ are piecewise constant in time and on each interval equal to a function from $V^u_m \times V^\varphi_m$. Note that $X_{m,l}\not\subseteq X_m$ since $X_{m,l}$ discontinuous in time. However, we have $X_{m,l} \subseteq L_2(0,T; V^u_m) \times L_2(0,T; V^\varphi_m)$.

We use the backward Euler scheme to define a fully discrete solution. Find a pair $(u_{m,l},\varphi_{m,l})=(g_u + \tilde u_{m,l}, g_\varphi + \tilde \varphi_{m,l})$ such that $(\tilde u_{m,l}, \tilde \varphi_{m,l}) \in X_{m, l}$ and for $n=1,...,N$,
\begin{subequations}\label{full_weak}
	\begin{align}
	\left\langle\frac{ u^n_{m,l}-u^{n-1}_{m,l}}{\tau_l}, v \right\rangle + \langle \nabla u^n_{m,l}, \nabla v \rangle &= -\langle \sigma(u^n_{m,l})\lceil \tilde \varphi^n_{m,l} \rceil \nabla \varphi^n_{m,l}, \nabla v \rangle  \label{fully_weak_u}\\&\quad + \langle \sigma(u^n_{m,l}) \nabla \varphi^n_{m,l} \cdot \nabla g^n_\varphi, v\rangle, \notag\\
	\langle \sigma(u^n_{m,l}) \nabla \varphi^n_{m,l}, \nabla w \rangle &= 0, \label{fully_weak_phi}\\
	\langle u^0_{m,l}, z \rangle &= \langle u_0, z \rangle, \label{fully_weak_initial}
	\end{align}
\end{subequations}
for all $(v,w)\in V^u_m\times V^\varphi_m$ and $z \in V^u_m$, where $u^n_{m,l} = u_{m,l}(t_n)$ and $\varphi^n_{m,l} = \varphi_{m,l}(t_n)$. 


\begin{lemma}\label{fully_bounds}
	A solution $(u_{m,l}, \varphi_{m,l})$ to \eqref{full_weak} fulfils the following bounds
	\begin{align}
		\|\nabla \tilde \varphi^n_{m,l} \|^2_{L_2(\Omega)} &\leq C(\sigma, g_\varphi),\label{fully_bounds_1}\\
		\|\tilde u^n_{m,l}\|^2_{L_2(\Omega)} + \int_0^{t_n} \|\nabla \tilde u_{m,l}\|^2_{L_2(\Omega)} \intd t &\leq C(u_0,\sigma,g_u,D_t g_u, g_\varphi),\label{fully_bounds_2}\\
		\sum_{n=1}^N \|\tilde u^n_{m,l}-\tilde u^{n-1}_{m,l}\|^2_{L_2(\Omega)} &\leq C(u_0,\sigma,g_u,D_t g_u,g_\varphi).\label{fully_bounds_3}
	\end{align}
	for $n=0,...,N$.
\end{lemma}
\begin{proof}
	Choosing $w=\tilde \varphi^n_{m,l}$ in \eqref{fully_weak_phi} we have $\|\nabla \tilde \varphi^n_{m,l}\|^2_{L_2(\Omega)} \leq C(\sigma)\|g^n_\varphi\|_{L_2(\Omega)}$. To prove \eqref{fully_bounds_2}, we note that
	\begin{align*}
		\langle \tilde u^n_{m,l}- \tilde u^{n-1}_{m,l}, \tilde u^n_{m,l} \rangle = \frac{1}{2}\|\tilde u^n_{m,l}\|^2_{L_2(\Omega)}-\frac{1}{2}\|\tilde u^{n-1}_{m,l}\|^2_{L_2(\Omega)} + \frac{1}{2}\|\tilde u^n_{m,l}-\tilde u^{n-1}_{m,l}\|^2_{L_2(\Omega)},
	\end{align*}
	and by choosing $v = \tau_l\tilde u^n_{m,l}$ in \eqref{fully_weak_u} and summing from $n=1$ to $N$ 
	\begin{align}\label{bounds_u_fully}
		&\frac{1}{2}\sum_{n=1}^N\big(\|\tilde u^n_{m,l}\|^2_{L_2(\Omega)}-\|\tilde u^{n-1}_{m,l}\|^2_{L_2(\Omega)} + \|\tilde u^n_{m,l}-\tilde u^{n-1}_{m,l}\|^2_{L_2(\Omega)}\big) \\&\quad+ \frac{1}{2}\int_0^T\|\nabla \tilde u_{m,l}\|^2_{L_2(\Omega)} \intd t \notag \\&\quad=- \sum_{n=1}^N\langle g^n_u-g^{n-1}_u, \tilde u^n_{m,l} \rangle - \int_0^T \langle \nabla g_u, \nabla \tilde u_{m,l} \rangle \intd t \notag\\&\qquad - \int_0^T \langle \sigma(u_{m,l})\lceil\tilde \varphi_{m,l}\rceil \nabla \varphi_{m,l}, \nabla\tilde u_{m,l} \rangle \intd t \notag \\&\qquad+ \int_0^T \langle \sigma(u_{m,l})\nabla \varphi_{m,l}\cdot \nabla g_\varphi, \tilde u_{m,l} \rangle \intd t =: I + II + III + IV. \notag 
	\end{align}
	For the first term we get
	\begin{align*}
		\sum_{n=1}^N \int_{t_{n-1}}^{t_n} \!\! \langle D_t g_u, \tilde u ^n_{m,l} \rangle \intd t &\leq C \sum_{n=1}^N \int_{t_{n-1}}^{t_n} \!\! \|D_tg_u\|^2_{H^1(\Omega;\Gamma^u_D)^\ast} \intd t+ \frac{1}{8}\int_0^T \!\! \|\nabla \tilde u_{m,l}\|^2_{L_2(\Omega)} \intd t.
	\end{align*}
The remaining terms $II-IV$ can be estimated as in the proof of Lemma~\ref{semi_bounds}. Using the telescoping effect of the first two terms in the sum in \eqref{bounds_u_fully} completes the proof.
\end{proof}

\begin{lemma}\label{fully_existence_uniqueness}
	There exists a solution $(\tilde u_{m,l}, \tilde \varphi_{m,l})\in X_{m,l}$ to \eqref{full_weak}. Furthermore, for any fixed $m$, there is an $L \in \N$, such that the solution $(\tilde u_{m,l}, \tilde \varphi_{m,l})$ is unique for any $l>L$.
\end{lemma}

\begin{proof}
	To prove this we use the ODE setting introduced in the proof of Lemma~\ref{semi_existence_uniqueness}.  Let $\tilde u^n_{m,l} = \sum_{j=1}^M \alpha^n_{l,j}\lambda_j$ for $n=1,...,N$. Then \eqref{full_weak} corresponds to finding $\alpha^n_l \in \R^M$ such that
	\begin{align}\label{ODE_alpha_BE}
	\frac{M\alpha^n_l - M\alpha^{n-1}_l}{\tau_l} + K\alpha^n_l = F(\alpha^n_l,t_n) + G(t_n),
	\end{align}
	for $n=1,..,N$, with $M\alpha^0_l = b$, where $b_i = \langle u_0-g_u(0), \lambda_i \rangle$, $F$ as in \eqref{ODE_alpha}, and $G$ slightly modified as
	\begin{align*}
	G(t_n) := -\langle \frac{g^n_u-g^{n-1}_u}{\tau_l}, \lambda_i \rangle - \langle \nabla g^n_u, \nabla \lambda_i \rangle.
	\end{align*}
	Note that \eqref{ODE_alpha_BE} is the backward Euler discretization of the ODE \eqref{ODE_alpha}.
	
	To apply Brouwer's fixed point theorem we define the mapping $f: \R^M \rightarrow \R^M$ such that $\beta = f(\gamma)$ is the solution to the system
	\begin{align*}
	\frac{M\beta - M\alpha^{n-1}_l}{\tau_l} + K\beta = F(\gamma,t_n) + G(t_n),
	\end{align*}
	which is equivalent to 
	\begin{align}\label{T_mapping}
	\beta = (M + \tau_lK)^{-1}(M\alpha^{n-1}_l + \tau_lF(\gamma,t_n) + \tau_lG(t_n)).
	\end{align}
	Let $\tilde y$ be the function corresponding to the vector $\gamma$, that is, $\tilde y^n = \sum_{j=1}^M \gamma^n_{j}\lambda_j$. From the definitions of $F_i$ and $G_i$ in the proof of Lemma~\ref{semi_existence_uniqueness}, and with $\tilde \psi=S(\tilde y)$ it follows that
	\begin{align*}
	|F_i(\gamma,t_n)|&\leq C(\sigma,g_\varphi)\|\nabla \psi\|_{L_2(\Omega)}\|\nabla \lambda_i\|_{L_2(\Omega)} \\&\quad+ \sigmaup \|\nabla \psi\|_{L_2(\Omega)}\|g_\varphi(t_n)\|_{L_3(\Omega)}\|\lambda_i\|_{L_6(\Omega)} \leq C_m(\sigma,g_\varphi),\\
	|G_i(t_n)|&\leq C_m(D_tg_u,g_u).
	\end{align*}
Here the boundedness of $\|\nabla \psi\|_{L_2(\Omega)}$ follows from Lemma~\ref{fully_bounds} with $\tilde u^n_{m,l}= \tilde y$. Hence, letting $B_R\in \R^M$ denote the ball with radius $R>0$, it is clear that $f\colon B_R\rightarrow B_R$ if $R$ sufficiently large. Now define $\beta_1=f(\gamma_1), \beta_2=f(\gamma_2) \in \R^M$. From \eqref{T_mapping} we have
	\begin{align*}
	(\beta^1-\beta^2) = \tau_l(M + \tau_lK)^{-1}(F(\gamma_1,t_n) - F(\gamma_2,t_n)),
	\end{align*}
	and using the Lipschitz continuity of $F(\cdot,t)$, see proof of Lemma~\ref{semi_existence_uniqueness}, we get
	\begin{align*}
	\|\beta^1-\beta^2\|_{\R^M} \leq C_L\tau_l\|(M + \tau_lK)^{-1}\|_{F}\|\gamma_1 - \gamma_2\|_{\R^M},
	\end{align*}
	where $\|\cdot\|_F$ denotes the matrix Frobenius norm which is finite since $M+\tau_l K$ is invertible. This proves that $f$ is continuous and the existence of a solution follows from Brouwer's fixed point theorem. Furthermore, it is clear that if $\tau_l$ is sufficiently small, or equivalently $l$ sufficiently large, then $f$ is a contraction on $\R^M$ and Banach's fixed point theorem gives uniqueness.
\end{proof}

\subsection{Convergence of fully discrete solutions}
To prove convergence of the fully discrete method we introduce the continuous and piecewise affine interpolant $\tilde U_{m,l}(t)$ of $\tilde u_{m,l}$
\begin{align}\label{U_def}
\tilde U_{m,l}(t) := \tilde u^{n-1}_{m,l} \frac{t_n-t}{t_n-t_{n-1}} + \tilde u^n_{m,l} \frac{t-t_{n-1}}{t_n-t_{n-1}}, \quad t \in I_n.
\end{align}
Note that
\begin{align}\label{DtU_equality}
D_t \tilde U_{m,l}(t) = \frac{\tilde u^n_{m,l} - \tilde u^{n-1}_{m,l}}{\tau_l}, \quad t\in I_n.
\end{align}
Using \eqref{fully_weak_u}, we have for $t\in I_n$ and $v\in H^1(\Omega;\Gamma^u_D)$
\begin{align*}
\langle D_t \tilde U_{m,l}(t),v\rangle &\leq \big(\|u^n_{m,l}\|_{H^1(\Omega)} + C(\sigma,g_\varphi)\|\nabla \varphi^n_{m,l}\|_{L_2(\Omega)} \\&\quad+ \sigmaup\|\nabla \varphi^n_{m,l}\|_{L_2(\Omega)} \|\nabla g^n_\varphi\|_{L_3(\Omega)} + \|\partial_t g^n_u\|_{H^1(\Omega;\Gamma^u_D)^\ast}\big)\|v\|_{H^1(\Omega)},
\end{align*}
where $\partial_t g^n_u = (g_u^n-g_u^{n-1})/\tau_l$. Note that $\partial_t g^n_u = \tau_l^{-1}\int_{I_n} D_t g_u \intd t$. This, together with Lemma~\ref{fully_bounds} and \ref{ass_g}, gives
\begin{align*}
\|D_t \tilde U_{m,l}\|_{L_2(0,T;(V^u_m)^\ast)} \leq C(u_0,\sigma,g_u,D_t g_u, g_\varphi).
\end{align*}
In addition, using \ref{ass_proj} as in the proof of Lemma~\ref{semi_bounds}, we get  
\begin{align}\label{fully_Dt_bound}
\|D_t \tilde U_{m,l}\|_{L_2(0,T;H^1(\Omega; \Gamma^u_D)^\ast)} \leq C(u_0,\sigma,g_u,D_t g_u, g_\varphi).
\end{align}

In the analysis we also use the following reformulation of \eqref{fully_weak_u} 
\begin{align}\label{fully_u_walkington}
\langle u^n_{m,l} - u^{n-1}_{m,l}, v \rangle = \langle F_{m,l}, v \rangle,
\end{align} 
with
\begin{align*}
\langle F_{m,l}, v \rangle &:= -  \tau_{l}\langle \nabla u^n_{m,l}, \nabla v \rangle - \tau_{l}\langle \sigma(u^n_{m,l})\lceil \tilde \varphi^n_{m,l} \rceil \nabla \varphi^n_{m,l}, \nabla v \rangle \\&\quad+ \tau_{l}\langle \sigma(u^n_{m,l}) \nabla \varphi^n_{m,l} \cdot \nabla g^n_\varphi, v\rangle, \quad \forall v \in V^u_m.
\end{align*}

\begin{thm}\label{fully_conv_galerkin}
A subsequence of solutions $(\tilde u_{m_k, l_k}, \tilde \varphi_{m_k,l_k}) \in X_{m_k,l_k}$ of \eqref{full_weak} converges strongly in $L_2(0,T; H^1(\Omega;\Gamma^u_D))\times L_2(0,T; H^1(\Omega;\Gamma^\varphi_D))$ to a solution $(\tilde u, \tilde \varphi)$ of \eqref{weak}.
\end{thm}

\begin{proof}
	From Lemma~\ref{fully_bounds}, \eqref{fully_Dt_bound}, and the reflexivity of the spaces $$L_2(0,T; H^1(\Omega;\Gamma^u_D)),\quad L_2(0,T; H^1(\Omega;\Gamma^\varphi_D)),\quad L_2(0,T;H^1(\Omega; \Gamma^u_D)^\ast)$$ there exists a subsequence such that
	\begin{align*}
		(\tilde u_{m_k,l_k}, \tilde \varphi_{m_k,l_k})  &\rightharpoonup (\tilde \vv, \tilde \psi) &&\quad \text{in } L_2(0,T; H^1(\Omega;\Gamma^u_D))\times L_2(0,T; H^1(\Omega;\Gamma^\varphi_D)), \\
		D_t \tilde U_{m_k,l_k} &\rightharpoonup D_t\tilde  U  &&\quad \text{in } L_2(0,T;H^1(\Omega; \Gamma^u_D)^\ast),
	\end{align*}
	for some $$(\tilde \vv, \tilde \psi) \in L_2(0,T; H^1(\Omega;\Gamma^u_D))\times L_2(0,T; H^1(\Omega;\Gamma^\varphi_D)), \ D_t\tilde U \in L_2(0,T;H^1(\Omega; \Gamma^u_D)^\ast).$$

	The convergence of the initial conditions follows as in the semi-discrete case and we conclude $u^0_{m_k,l_k} \rightarrow \vv (0) = u(0)$. 
	
	Next, we prove that $D_t \tilde \vv = \tilde D_t U$. First, we note that $\tilde U_{m_k,l_k}-\tilde u_{m_k,l_k} \rightharpoonup 0$ weakly in $L_2(0,T;L_2(\Omega))$. To see this, pick $\chi_{[\bar \tau a,\bar \tau b]} \otimes v : (t,x) \mapsto \chi_{[\bar \tau a,\bar \tau b]}(t) \, v(x)$, for $v\in L_2(\Omega)$, $a<b$, and some $\bar\tau>0$. Now, due to \cite[Theorem 8.9]{Roubicek13}, for $\tau_{l_k}\leq \bar\tau$
	\begin{align*}
	&\int_0^T \langle \tilde U_{m_k,l_k}-\tilde u_{m_k,l_k}, \chi_{[\bar \tau a,\bar \tau b]} \otimes v \rangle \intd t = \frac{\tau_{l_k}}{2}\langle \tilde u_{m_k,l_k}^{\bar{\tau}b/\tau_{l_k}} - \tilde u_{m,l_k}^{\bar{\tau}a/\tau_{l_k}},v\rangle \leq C\tau_{l_k} \rightarrow 0,
	\end{align*}
	where we used \eqref{U_def} and the bounds in Lemma~\ref{fully_bounds}. Since $\tilde{U}_{m_k}$ is bounded in $L_2(0;T;L_2(\Omega))$ and functions of the form $\chi_{[\bar \tau a,\bar \tau b]} \otimes v$ are dense in $L_2(0,T;L_2(\Omega))$ this implies $\tilde U_{m_k,l_k}-\tilde u_{m_k,l_k} \rightharpoonup 0$ in $L_2(0,T;L_2(\Omega))$, see \cite[Theorem~3~p.~121]{Yosida95}. Thus we get, for $v\in C^1(0,T;H^1(\Omega;\Gamma^u_D))$ with $v(0)=v(T)=0$, 
	\begin{align*}
	\int_0^T \langle D_t \tilde U, v \rangle \intd t \leftarrow \int_0^T \langle D_t  \tilde U_{m_k,l_k}, v \rangle \intd t = -\int_0^T \langle  \tilde U_{m_k,l_k}, D_t v \rangle \intd t \rightarrow -\int_0^T \langle \tilde \vv, D_t v \rangle \intd t,
	\end{align*}
	and we conclude $D_t \tilde \vv = D_t\tilde U$, due to \eqref{fully_Dt_bound} and since $C^1(0,T;H^1(\Omega;\Gamma^u_D))$ with $v(0)=v(T)=0$ is dense in $L_2(0,T;L_2(\Omega))$, see see \cite[Theorem~3~p.~121]{Yosida95}. This implies that $(\tilde \vv, \tilde \psi) \in X$.
	
	In \cite[Theorem 3.1]{Walkington10} it is proved that if $\{u_{m_k,l_k}\}_k$ and $\{F_{m_k,l_k}\}_k$ in \eqref{fully_u_walkington} are bounded in $L_2(0,T;H^1(\Omega;\Gamma^u_D))$ and $L_2(0,T; (V^u_{m_k})^\ast)$, respectively, uniformly in $k$, then $\{u_{m_k,l_k}\}$ is precompact in $L_2(0,T; L_2(\Omega))$. Here, the boundedness of $\{u_{m_k,l_k}\}$ and $\{F_{m_k,l_k}\}$ follows from Lemma~\ref{fully_bounds}, \ref{ass_g}, and the bounds on $\sigma$ and $\lceil \cdot \rceil$. Hence, there exists a subsequence, still denoted $(m_k,l_k)$, such that $\tilde u_{m_k,l_k} \rightarrow \tilde \vv$ strongly and pointwise a.e.~in $L_2(0,T;L_2(\Omega))$.
	
	Owing to Lemma~\ref{lem_second_equation} we have for some subsequence, still denoted $(m_k,l_k)$, that $\varphi_{m_k,l_k}, \nabla \varphi_{m_k,l_k}, \sigma(u_{m_k,l_k})\nabla \varphi_{m_k,l_k}$, and $\sigma(u_{m_k,l_k})\nabla \lceil \tilde \varphi_{m_k,l_k} \rceil \varphi_{m_k,l_k}$, converges strongly in $L_2(0,T;L_2(\Omega))$. Using \eqref{fully_weak_u} we get for $v \in L_2(0,T;V^u_m)$
	\begin{align*}
	&\int_0^T \langle D_t \tilde U_{m_k,l_k} + \partial_t g_u, v \rangle + \langle \nabla u_{m_k,l_k}, \nabla v \rangle \intd t \\&\quad= \int_0^T \big(- \langle \sigma(u_{m_k,l_k})\lceil \tilde \varphi_{m_k,l_k} \rceil \nabla \varphi_{m_k,l_k}, \nabla v \rangle + \langle \sigma(u_{m_k,l_k})\nabla \varphi_{m_k,l_k}\cdot \nabla g_\varphi, v \rangle\big) \intd t, 
	\end{align*}
	where $\partial_t g^n_u = (g^n_u-g^{n-1}_u)/\tau_l$. Keeping $v$ fixed we get as $k \to \infty$
	\begin{align*}
	\int_0^T \langle D_t \tilde y + D_t g_u, v \rangle + \langle \nabla u, \nabla v \rangle \intd t &= \int_0^T \big(- \langle \sigma(u_{m_k})\lceil \tilde \varphi_{m_k} \rceil \nabla \varphi_{m_k}, \nabla v \rangle \\&\quad+ \langle \sigma(u_{m_k})\nabla \varphi_{m_k}\cdot \nabla g_\varphi, v \rangle\big) \intd t, 
	\end{align*}
	where we used the weak convergence of $\tilde U_{m_k}$ and $\nabla u_{m_k}$, the strong convergence of $\sigma(u_{m_k,l_k})\nabla \varphi_{m_k,l_k}$ and $\sigma(u_{m_k,l_k})\nabla \lceil \tilde \varphi_{m_k,l_k} \rceil \varphi_{m_k,l_k}$, and $\partial_t g_u \to D_t g_u$. Now, due to the density of $\{v(x,t): v|_{I_n} \in V^u_m\}$ and the boundedness of $u_{m_k,l_k}, \varphi_{m_k,l_k}$, this holds for all $v \in L_2(0,T;H^1(\Omega;\Gamma^u_D))$ \cite[Theorem~3~p.~121]{Yosida95}. Furthermore, in the spirit of Lemma~\ref{second_equation} we may prove that $\tilde \psi$ solves \eqref{weakphi}. Hence, $(\tilde \vv, \tilde \psi)$ is a solution to \eqref{weak}.
	
	To prove that $\tilde u_{m_k,l_k} \rightarrow \tilde \vv$ strongly in $L_2(0,T; H^1(\Omega;\Gamma^u_D))$ we may mimic the argument in the proof of Theorem~\ref{semi_conv_galerkin} (recall \eqref{DtU_equality}).
\end{proof}

\begin{corollary}\label{cor_fully_discrete}
	If the solution $(u,\varphi)$ to \eqref{weak} is unique, then the whole sequence of fully discrete approximations $(\tilde u_{m,l},\tilde \varphi_{m,l}) \in X_{m,l}$ converges.
\end{corollary}

\begin{proof}
	If the solution to \eqref{weak} is unique there can only be one accumulation point, cf.~Corollary~\ref{cor_semi_conv}.
\end{proof}

\begin{remark}
	There are many other time discretizations available, see, e.g., \cite[Chapter~16]{Thomee2006} for discretizations of a parabolic nonlinear problem. For instance, one could use linearized schemes to avoid solving a nonlinear problem in each time step. In this case the existence and uniqueness result follows easily. Convergence may be deduced by comparing the linearized solution to the backward Euler solution. However, to avoid overloading the paper, we shall not analyze this here. 
\end{remark}


\section{Regularity and uniqueness}\label{sec_regularity}
In this section we prove additional regularity and uniqueness of a solution to the weak problem \eqref{weak}. For this purpose we use Theorem~\ref{thm_main_reg} below, which is based on \cite{Hieber08, Meinlschmidt17}, see, in particular, \cite[Theorem 3.1]{Hieber08}. The theory in \cite{Mitrea07} gives a setting where assumption \ref{cond_OP} below, is satisfied. Our aim is to combine these results to obtain additional regularity for the Joule heating problem with mixed boundary conditions on creased domains. For similar settings, see \cite[Section 3]{Meinlschmidt17}, where regularity for the Joule heating problem with pure Robin boundary conditions for the temperature and mixed boundary conditions for the potential is studied. In addition, we also prove higher regularity of the solution in the interior of the domain. We emphasize that the differences in the regularity within the domain makes the problem well suited for $h$- and $hp$-adaptive finite elements.

In \cite{Hieber08} the following type of systems are studied
\begin{subequations}\label{eq_hieber}
	\begin{alignat}{2}
	D_t u - \Delta u &= R(u),& \quad & \text{in }\Omega \times (0,T),\\
	u &= g_u,& & \text{on }\Gamma^u_D \times (0,T) \\
	n \cdot \nabla u &= 0,& & \text{on }\Gamma^u_N \times (0,T),\\
	u(\cdot, 0) &= u_0, & & \text{in }\Omega.		
	\end{alignat}
\end{subequations}
If we define $R(u) = \sigma(u)|\nabla S(u)|^2$ such that $\varphi = S(u)$ solves
\begin{subequations}\label{S_hieber}
	\begin{alignat}{2}
	-\nabla \cdot (\sigma(u)\nabla\varphi) &= 0, & \quad & \text{in }\Omega \times (0,T),\label{S_hieber_phi}\\
	\varphi &= g_\varphi,& & \text{on }\Gamma^\varphi_D \times (0,T) \\
	n \cdot \nabla \varphi &= 0,& & \text{on }\Gamma^\varphi_N \times (0,T),
	\end{alignat}
\end{subequations}
then \eqref{eq_hieber} is equivalent to \eqref{jouleheating}.

For $p>\frac{3}{2}$ and a fixed $r > \frac{4p}{2p-3}$ we consider the following assumptions.
\begin{enumerate}[label=(B\arabic*)]
	\item $\Omega$ is a bounded domain with Lipschitz boundary (in the sense of Gr{\"o}ger, see \cite{Hieber08}), $\meas(\Gamma^u_D)> 0$, and   $\meas(\Gamma^\varphi_D)>0$.\label{ass_omega_reg}
	\item $g_u \in C([0,T]; W^1_{2p}(\Omega))\cap W^1_r(0,T;L_p(\Omega))$, $\Delta g_u(t) = 0$, $t\in(0,T)$, and $g_\varphi \in L_r(0,T; W^1_{2p}(\Omega))$. \label{ass_g_reg}
	\item $u_0-g_u(0) \in (L_p(\Omega),D(\Delta_p))_{1-\frac{1}{r},r}$. \label{ass_intial_reg}
	\item $\sigma \in C^1(\R)$, Lipschitz continuous, and $0< \sigmalow \leq \sigma(x) \leq \sigmaup < \infty$, $\forall x \in \R$. \label{ass_sigma_reg} 
	\item $\Delta$ is a topological isomorphism from $W^1_{2p}(\Omega;\Gamma^u_D)$ onto $W^{-1}_{2p}(\Omega;\Gamma^u_D)$ and from $W^1_{2p}(\Omega;\Gamma^\varphi_D)$ onto $W^{-1}_{2p}(\Omega;\Gamma^\varphi_D)$. \label{cond_OP}
\end{enumerate}

Here $D(A)$ denotes the domain of the operator $A$, that is 
\begin{align*}
	D(A) = \{v \in H^1(\Omega;\Gamma_D): \exists f \in L_2(\Omega), \ -a( v,  w) = \langle f, w \rangle, \ \forall w \in H^1(\Omega;\Gamma_D)\},
\end{align*}
for $\Gamma_D \subseteq \partial \Omega$. The semigroup generated by $A$ extends to a $C_0$-semigroup on $L_p(\Omega)$, $1< p < \infty$, and we denote its generator by $A_p$, see \cite{Hieber08} and references therein. In particular, we use $\Delta_p$ for the Laplacian that maps onto $L_p(\Omega)$. Furthermore, for Banach spaces $V$ and $W$ forming an interpolation couple, $(V,W)_{\alpha,\beta}$ denotes the real interpolation space.

The following theorem relies on \cite[Theorem~3.1]{Hieber08}. 

\begin{thm}\label{thm_main_reg}
	Let $p>\frac{3}{2}$ and $r>\frac{4p}{2p-3}$. Under the assumptions \ref{ass_omega_reg}-\ref{cond_OP}, there exists a unique solution to \eqref{jouleheating} satisfying
	\begin{align}\label{u_reg}
		\tilde u \in W^1_r(0,T_\ast; L_p(\Omega)) \cap L_r(0,T_\ast;W^1_{2p}(\Omega;\Gamma^u_D)), \quad \tilde \varphi \in L_r(0,T_\ast;W^1_{2p}(\Omega;\Gamma^\varphi_D)), 
	\end{align}
	for some $0<T_\ast\leq T$.
\end{thm}

\begin{proof}
	Consider
	\begin{enumerate}[label=(B\arabic*)]\setcounter{enumi}{5}
		\item The function $R\colon W^1_{2p}(\Omega) \rightarrow L_p(\Omega)$ is continuous. \label{cond_Ra}
		\item $R(0) \in L_r(0,T; L_p(\Omega))$ and for $\beta>0$ there exist $g_\beta \in L_r(0,T)$ such that
		\begin{align*}
		\|R(u_1)-R(u_2)\|_{L_p(\Omega)} \leq g_\beta(t)\|u_1-u_2\|_{W^1_{2p}(\Omega)}, \quad t\in(0,T),
		\end{align*} 
		provided $\max(\|u_1\|_{W^1_{2p}(\Omega)}, \|u_2\|_{W^1_{2p}(\Omega)})\leq \beta$.\label{cond_Rb}
	\end{enumerate}
	In \cite[Theorem~3.1]{Hieber08} it is proved that if the conditions \ref{ass_omega_reg}-\ref{cond_OP} together with \ref{cond_Ra}-\ref{cond_Rb} on the operator $R$ are satisfied, then there is a unique solution to \eqref{eq_hieber} satisfying $u \in W^1_r(0,T_\ast; L_p(\Omega)) \cap L_r(0,T_\ast; D(\Delta_p))$ for some $0<T_\ast\leq T$. 
	
	Our aim is now to prove that if $R(u)=\sigma(u)|\nabla S(u)|^2$ with $S:W^1_{2p}(\Omega) \rightarrow W^1_{2p}(\Omega)$ defined as in \eqref{S_hieber}, then $R$ satisfies \ref{cond_Ra}-\ref{cond_Rb} and we may conclude that $(u,\varphi)$ solves \eqref{jouleheating} and \eqref{u_reg} is fulfilled.
	
	From \cite[Corollary 3.24]{Meinlschmidt17} it follows that the operator $(-\nabla \phi \cdot \nabla)^{-1}$ is a linear homeomorphism from $W^{-1}_{2p}(\Omega;\Gamma^\varphi_D)$ to $W^{1}_{2p}(\Omega;\Gamma^\varphi_D)$, if $\phi \in \mathfrak{C}$ is uniformly continuous on $\bar \Omega$, $\mathfrak{C}$ compact set in $C(\bar\Omega)$, $\phi$ admits a positive lower bound, and \ref{cond_OP} holds. It also holds that $(-\nabla\cdot \phi \nabla)^{-1}$ is Lipschitz with respect to $\phi$.
	
	In our setting, we have due to Morrey's inequality $W^1_{2p}(\Omega) \subseteq C^{0,\alpha}(\bar \Omega)$, for some $\alpha>0$, and $C^{0,\alpha}(\bar \Omega)$ embeds compactly into $C(\bar \Omega)$. This implies that $\mathfrak C = \{\sigma(u): u\in W^1_{2p}(\Omega)\}$ is compact and $\sigma(u)$ is uniformly continuous, since $\sigma$ is Lipschitz continuous. Furthermore, $\sigma(\cdot) \geq \sigmalow > 0$. Hence, given $\phi=\sigma(u)$ we deduce that there exists a unique solution $\tilde \varphi(t) \in W^1_{2p}(\Omega;\Gamma^\varphi_D)$ to \eqref{S_hieber}. This proves that the mappings $S$ and $R$ are well defined in the given spaces.
	
	Given $u$ fixed, let $F:W^{-1}_{2p}(\Omega;\Gamma^\varphi_D)\rightarrow W^{1}_{2p}(\Omega;\Gamma^\varphi_D)$, such that $\psi = F(g)$ is the solution to 
	\begin{align*}
	\nabla\cdot (\sigma(u) \nabla \psi) = g, \quad \psi|_{\Gamma^\varphi_D} = 0.
	\end{align*}	
	By letting $G = \nabla \cdot (\sigma(u)\nabla g_\varphi)$ and using $\varphi = \tilde \varphi + g_\varphi$  it follows from \eqref{S_hieber_phi} that $\tilde \varphi = F(G)$. Since the operator $F$ is bounded we get
	\begin{align}\label{reg_time}
	\|\tilde \varphi\|_{W^1_{2p}(\Omega;\Gamma^\varphi_D)} &\leq C\|G\|_{W^{-1}_{2p}(\Omega;\Gamma^\varphi_D)} = C \sup_{w \in W^1_{(2p)'}(\Omega;\Gamma^\varphi_D)\setminus \{0\}} \frac{\langle \nabla \cdot (\sigma(u)\nabla g_\varphi),w\rangle}{\|w\|_{W^1_{(2p)'}(\Omega)}} \\&= C \sup_{w \in W^1_{(2p)'}(\Omega;\Gamma^\varphi_D)\setminus \{0\}} \frac{\langle \sigma(u)\nabla g_\varphi,\nabla w\rangle}{\|w\|_{W^1_{(2p)'}(\Omega)}} \leq C\|g_\varphi\|_{W^1_{2p}(\Omega)}, \notag
	\end{align}
	where $(2p)'$ is the H\"{o}lder conjugate exponent to $2p$. Thus, $\varphi\in L_r(0,T;W^1_{2p}(\Omega))$, since $g_\varphi \in L_r(0,T;W^1_{2p}(\Omega))$, which implies $R(0) \in L_r(0,T;L_p(\Omega))$ in \ref{cond_Rb}.
	
	For $\varphi_1 = S(u_1)$ and $\varphi_2 = S(u_2)$ we get
	\begin{align*}
	&\|R(u_1)-R(u_2)\|_{L_p(\Omega)} \\&\quad\leq \|(\sigma(u_1)-\sigma(u_2))|\nabla \varphi_1|^2\|_{L_p(\Omega)} + \|\sigma(u_2)(|\nabla \varphi_1|^2-|\nabla \varphi_2|^2\|_{L_p(\Omega)} \\
	&\quad\leq C(\|u_1-u_2\|_{L_\infty(\Omega)}\|\nabla \varphi_1\|^2_{L_{2p}(\Omega)} + \|\nabla(\varphi_1 + \varphi_2)\|_{L_{2p}(\Omega)}\|\nabla(\varphi_1 - \varphi_2)\|_{L_{2p}(\Omega)}).
	\end{align*}
	Using Sobolev's inequality we get $\|u_1-u_2\|_{L_\infty(\Omega)} \leq C\|u_1-u_2\|_{W^1_{2p(\Omega)}}$. Due to the Lipschitz property of $(-\nabla\cdot \sigma(u) \nabla)^{-1}$ we get
	\begin{align*}
	\|\tilde \varphi_1 - \tilde \varphi_2\|_{W^1_{2p}(\Omega)} \leq C\|\sigma(u_1) - \sigma(u_2)\|_{L_\infty} \leq C\|u_1-u_2\|_{L_\infty} \leq C\|u_1-u_2\|_{W^1_{2p}(\Omega)}.
	\end{align*}
	which proves \ref{cond_Rb} and the continuity in \ref{cond_Ra}. Hence there exists a unique solution $u \in W^1_r(0,T_\ast; L_p(\Omega)) \cap L_r(0,T_\ast; D(\Delta_p))$ for some $T_\ast > 0$.
	
	Finally, by definition, $D(\Delta_p)$ denotes the domain such that the Laplacian maps into $L_p(\Omega)$. Let $p'$ and $(2p)'$ be the H\"{o}lder conjugates to $p$ and $2p$, respectively. Then
	\begin{align*}
	L_p(\Omega)=(L_{p'}(\Omega))^\ast, \quad W^{-1}_{2p}(\Omega; \Gamma^u_D) = (W^1_{(2p)'}(\Omega; \Gamma^u_D))^\ast.
	\end{align*}
	By Sobolev's inequality we have,
	\begin{align*}
	W^1_{(2p)'}(\Omega) \subseteq L_{6p/(4p-3)}(\Omega),
	\end{align*}
	and since $p'=p/(p-1)<6p/(4p-3)$ when $p>3/2$ we conclude $W^1_{(2p)'}(\Omega; \Gamma^u_D) \subseteq L_{p'}(\Omega)$, or equivalently $L_p(\Omega) \subseteq W^{-1}_{2p}(\Omega; \Gamma^u_D)$. Now, since we know that $\Delta$ is an isomorphism between $W^1_{2p}(\Omega;\Gamma^u_D)$ and $W^{-1}_{2p}(\Omega;\Gamma^u_D)$ and $L_p(\Omega) \subseteq W^{-1}_{2p}(\Omega;\Gamma^u_D)$ we deduce $D(\Delta_p) \subseteq W^1_{2p}(\Omega;\Gamma^u_D)$ and \eqref{u_reg} follows.
\end{proof}

Note that the result is only local in time, that is, the additional regularity and uniqueness are only guaranteed up to some $T_\ast\leq T$.

We provide an example of a geometric setting for which \ref{cond_OP} is satisfied. Assume instead of \ref{ass_omega_reg} the following
\begin{enumerate}[label=(B\arabic*')]
	\item $\Omega$ is a creased domain with respect to the boundary conditions for $u$ and $\varphi$. In addition $\meas(\Gamma^u_D)>0$ and $\meas(\Gamma^\varphi_D)>0$.\label{ass_creased}
\end{enumerate}
For the full definition of creased domains we refer to \cite[Definition~2.3]{Mitrea07}. We note however, that in our setting, it implies that $\Omega$ is a Lipschitz domain with $\Gamma^u_D$ and $\Gamma^\varphi_D$ open and non-empty and $\partial \Gamma^u_D$ and $\partial \Gamma^\varphi_D$ are not re-entrant. This means that the angles between the Dirichlet and Neumann parts of the boundary are strictly less than $\pi$.

\begin{corollary}
	Let $p>\frac{3}{2}$ and $r>\frac{4p}{2p-3}$. Under the assumptions \ref{ass_creased} and \ref{ass_g_reg}-\ref{ass_sigma_reg}, there exists a unique solution to \eqref{jouleheating} satisfying
	\begin{align*}
	\tilde u \in W^1_r(0,T_\ast; L_p(\Omega)) \cap L_r(0,T_\ast;W^1_{2p}(\Omega;\Gamma^u_D)), \quad \tilde \varphi \in L_r(0,T_\ast;W^1_{2p}(\Omega;\Gamma^\varphi_D)), 
	\end{align*}
	for some $0<T_\ast\leq T$.
\end{corollary}

\begin{proof}
To see that condition \ref{cond_OP} is fulfilled we use the result on equations of Poisson type in \cite{Mitrea07}. If $\Omega$ is a creased domain, $\Gamma_D \subseteq \partial \Omega$, and $g \in B^s_{q,q}(\Gamma_D)$, then there exists $\epsilon \in (0,\frac{1}{2})$ such that Poisson's equation is well-posed in the spaces
\begin{align}\label{Mitrea_eq}
v \in W^{s+\frac{1}{q}}_q(\Omega; \Gamma_D), \ -\Delta v = f \in \big(W^{2-s-\frac{1}{q}}_{q'}(\Omega; \Gamma_D)\big)^\ast, \ v|_{\Gamma_D} = g_u,
\end{align}
for $\frac{1}{q'}=1-\frac{1}{q}$ and $(s,\frac{1}{q}) \in \mathcal{H}_\epsilon$ where $\mathcal{H}_\epsilon$ is the polygon with vertices in 
\begin{align*}
(0,0), \quad (\epsilon, 0), \quad (1,\frac{1}{2}-\epsilon), \quad (1,1), \quad (1-\epsilon,1), \quad (0,\frac{1}{2}+\epsilon).
\end{align*}
Choosing $s=\frac{8+\epsilon}{12}$ and $q = \frac{12}{4-\epsilon}$ we have for $p=\frac{6}{4-\epsilon}>\frac{3}{2}$
\begin{align*}
W^{s+\frac{1}{q}}_q = W^1_{\frac{12}{4-\epsilon}} = W^1_{2p} \ \text{ and } (\ W^{2-s-\frac{1}{q}}_{q'})^\ast= (W^1_{\frac{12}{8+\epsilon}})^\ast = W^{-1}_{2p},
\end{align*}
and since $W^1_{2p}(\Omega)|_{\Gamma_d} = B^s_{p,p}(\Gamma_D)$ assumption \ref{ass_g_reg} gives $g_u(t), g_\varphi(t) \in B^s_{q,q}(\Gamma^u_D)$. We conclude that \ref{cond_OP} holds.
\end{proof}

\begin{remark}
There are other geometric settings where condition \ref{cond_OP} is fulfilled, see, for instance,  \cite{Hieber08,Meinlschmidt17,Disser15}.
\end{remark}

\begin{remark}
In \cite[Section 4]{Antontsev94} it is established that $\nabla \varphi \in L_{2q/(q-3)}(0,T;L_q(\Omega))$, for $q>3$, is sufficient for a unique solution, see also Theorem~\ref{existence_uniqueness_classical}. This agrees with the regularity we get in Theorem~\ref{thm_main_reg}.
\end{remark}

The next theorem shows that higher regularity achieved in the interior of the domain, cf.~the stationary case \cite[Theorem~4.2]{Jensen13}. Here we use the notation $D(\Delta_{p,k})$ for the domain such that $\Delta$ maps into $W^k_p(\Omega)$. Note that $D(\Delta_{p}) = D(\Delta_{p,0})$.

\begin{thm}\label{thm_interior_reg}
	Let $0<T_0<T_\ast$ and let $\Omega_0$ be a relatively compact domain in $\Omega$: $\Omega_0 \Subset \Omega$. Let $(u,\varphi)$ be the solution to \eqref{jouleheating} such that
	\begin{align}\label{u_phi_reg}
	\tilde u \in W^1_r(0,T_\ast; L_p(\Omega)) \cap L_r(0,T_\ast;W^1_{2p}(\Omega;\Gamma^u_D)), \quad \tilde \varphi \in L_r(0,T_\ast;W^1_{2p}(\Omega;\Gamma^\varphi_D)), 
	\end{align}
	for some $p>\frac{3}{2}$ and $r>\frac{4p}{2p-3}$. Then $\tilde u,\tilde \varphi \in L_r(T_0,T_\ast; W^2_s(\Omega_0))$ for all $s\in(1,\infty)$. If $\sigma \in C^\infty(\R)$, then $\tilde u, \tilde \varphi \in L_r(T_0,T_\ast;C^\infty(\Omega_0))$.
\end{thm}

\begin{proof}
	Let $\Omega_\infty$ and $\{\Omega_i\}_{i=1}^\infty$ be smooth domains such that $\Omega_{i-1} \Subset \Omega_{i}$ and $\Omega_i \subset \Omega_\infty \Subset \Omega$, for $i=0,1,...$. Assume, without loss of generality, that the boundary data $g_u$ and $g_\varphi$ have smooth extensions to $\Omega_\infty$ such that $g_u, D_t g_u, g_\varphi \in L_r(0,T;C^\infty(\Omega_\infty))$. Without loss of generality, we also assume $p=6/(4-\epsilon)$ for some $\epsilon>0$.
	
	Let $\bar\zeta_i \in C^\infty(\Omega_i,[0,1])$, such that $\bar\zeta_i|_{\partial \Omega_i}=0$ and $\bar\zeta_i|_{\Omega_{i-1}}=1$. Furthermore, let $\{T_i\}_{i=1}^\infty$ and $T_0$ be positive numbers such that $T_i<T_{i-1}$, and $0< T_i < T_0 < T_\ast$ for all $i$. Define $\eta_i(t) \in C^\infty([T_i,T_\ast],[0,1])$ such that $\eta_i(T_i)=0$ and $\eta_i|_{[T_{i-1}, T_0]}=1$. Let $\zeta_i := \eta_i\bar\zeta_i$, then $(\zeta_i\tilde u,\zeta_i \tilde \varphi)$ satisfies the following system in $\Omega_i \times (T_i,T_\ast)$.
	\begin{subequations}\label{smooth_system}
		\begin{align}
		D_t (\zeta_i\tilde u) - \Delta (\zeta_i\tilde u) &= \zeta_i\sigma(u) |\nabla \varphi|^2 -\zeta_i(D_t g_u - \Delta g_u) + \tilde uD_t\zeta_i -2\nabla \zeta_i \cdot \nabla \tilde u \\&\quad- \tilde u \Delta \zeta_i,\notag \\
		\Delta (\zeta_i\tilde \varphi) &= \zeta_i\frac{\sigma'(u)}{\sigma(u)}\nabla u \cdot \nabla \varphi - \zeta_i\Delta g_\varphi - 2\nabla \zeta_i \cdot \nabla \tilde \varphi - \tilde \varphi \Delta \zeta_i,
		\end{align}
	\end{subequations}
	with homogeneous Dirichlet conditions and zero initial data. Note that we have used 
	\begin{align*}
	\nabla \cdot(\sigma(u)\nabla \varphi) = 0 \Leftrightarrow \Delta \varphi= \frac{\sigma'(u)}{\sigma(u)}\nabla u \cdot \nabla \varphi,
	\end{align*}
	in the second equation. Because of the assumed regularity in \eqref{u_phi_reg} and the smoothness of $\zeta_i,g_u,$ and $g_\varphi$, the right-hand sides in \eqref{smooth_system} are in $L_r(T_i,T_\ast;L_p(\Omega_i))$. 
	
	There exists a unique solution in $W^2_{p}(\Omega_i)$ to Poisson's equation with homogeneous Dirichlet boundary conditions, if the domain is smooth and the right-hand side in $L_p(\Omega_i)$, $1<p<\infty$, see e.g.~\cite[Theorem~9.15]{Gilbarg01}. We conclude that, for a fixed $t$, $\zeta_i(t) \tilde \varphi(t) \in W^2_p(\Omega_i)$. We may now use elliptic regularity in $L_p$, see \cite[Lemma 9.17]{Gilbarg01}, to deduce
	\begin{align*}
	\|\zeta_i \tilde \varphi\|_{W^2_{p}(\Omega_i)} \leq C\|\zeta_i\frac{\sigma'(u)}{\sigma(u)}\nabla u \cdot \nabla \varphi - \zeta_i\Delta g_\varphi - 2\nabla \zeta_i \cdot \nabla \tilde \varphi - \tilde \varphi \Delta \zeta_i\|_{L_p(\Omega_i)}.
	\end{align*} 
	The regularity in time of the right hand side implies $\zeta_i\tilde \varphi \in L_r(T_i,T_\ast;W^2_p(\Omega_{i}))$. Thus $\tilde \varphi \in L_r(T_{i-1},T_\ast;W^2_p(\Omega_{i-1}))$, since $\zeta_i = 1$ on $[T_{i-1},T_\ast]\times \Omega_{i-1}$. 
	 
	For the parabolic equation we use the theory for maximal $L_p$-regularity with homogeneous Dirichlet boundary conditions on smooth domains, see, e.g., \cite[Theorem 3.1]{Hieber97}. If the right-hand side is in $L_r(0,T;L_p(\Omega))$ and the initial data is zero, then the solution belongs to $L_r(0,T;D(\Delta_p))\cap W^1_r(0,T;L_p(\Omega))$. From the results on Poisson's equation we deduce $D(\Delta_p)\subset W^2_p(\Omega_i)$ and thus $\tilde u\in L_r(T_{i-1},T_\ast;W^2_{p}(\Omega_{i-1}))\cap W^1_r(T_{i-1},T_\ast;L_p(\Omega_{i-1}))$.
	 
	From the Sobolev inequality we have $W^2_{p}(\Omega_{i-1}) \subseteq W^1_{3p/(3-p)}(\Omega_{i-1})$. Using that $2p=12/(4-\epsilon)$ for some $\epsilon>0$ we get $3p/(3-p) = 12/(4-2\epsilon)$. Hence, we can substitute $\epsilon$ by $2\epsilon$, pass from $i$ to $i-1$ and repeat the argument. Note that if $12/(4-\epsilon)$ becomes negative, the right-hand side is in $L_\infty(\Omega_i)$. By induction $\tilde u, \tilde \varphi \in L_r(T_0,T_\ast;W^2_s(\Omega_0))$ for any $s\in(1,\infty)$.
	
	Now assume $\sigma \in C^\infty(\R)$. A solution to Poisson's equation on a smooth domain is in $W^{k+2}_p(\Omega_i)$ if the right-hand side is in $W^k_p(\Omega_i)$, see e.g.~\cite[Theorem~9.19]{Gilbarg01}. By applying Leibniz's rule it is clear that there is an $s'$ such that the right-hand sides in \eqref{smooth_system} belongs to $L_r(T_i,T_\ast;W^{k}_s(\Omega_i))$ if $\tilde \varphi, \tilde u \in L_r(T_i,T_\ast;W^{k+1}_{s'}(\Omega_i))$. Hence, we may perform induction over $k$ and pass from $i$ to $i-1$, to achieve $\tilde u, \tilde \varphi \in L_r(T_0,T_\ast; W^{k+2}_s(\Omega_0))$, for any $k,s >1$. This implies $\tilde u, \tilde \varphi \in L_r(T_0,T_\ast; C^\infty(\Omega_0))$.
\end{proof}


\section{Numerical Examples}\label{sec:examples}
In this section we consider four different examples. The first two are designed to test the convergence rates for different settings. In the first example we choose the domain and the data such that the exact solution is known. To achieve this we add a function $f(x,t)$ to the right-hand side in \eqref{weaku} and consider non-zero Neumann data for $\phi$, see Subsection~\ref{sec:example1} below. For the second example we consider a setting that does not fulfil the creased domain conditions. For this problem we expect low regularity and reduced convergence rates. Finally, in the last two examples we test a goal oriented adaptivity method.

In all cases we consider a continuous, piecewise affine finite element discretization. We let $\{\mathcal T_{m}\}_m$ denote a family of uniform triangulations of the domain such that $h_{m+1} = 2^{-1}h_m$, $h_0 \in \R$, where $h_m$ is the maximal mesh size on $\mathcal T_m$. With this notation we may define
\begin{align*}
V^u_m &:= \{v \in H^1(\Omega;\Gamma^u_D)\cap C^0(\bar \Omega): v|_K \ \text{is a polynomial of degree} \leq 1, \forall K \in \mathcal T_h \},\\
V^\varphi_m &:= \{v \in H^1(\Omega;\Gamma^\varphi_D)\cap C^0(\bar \Omega): v|_K \ \text{is a polynomial of degree} \leq 1, \forall K \in \mathcal T_h \}.
\end{align*} 
For the time discretization, we let $\tau_l = 2^{-l}T$ and the fully discrete space $X_{m,l}$ is defined as in Section~\ref{sec_discrete}. 

In the first two experiments we keep the time step proportional to the mesh size in each refinement. That is, we consider spaces of the form $X_{k,k}$, for $k=1,2,3,...$. This means that if the solution has sufficient regularity, then we expect at most linear convergence rate in the norm $L_2(0,T; H^1(\Omega))$, see also \cite{Elliott95,Stillfjord17,Gao14}.   

All computations are made using the FEniCS software \cite{FenicsBook}.

\subsection{Example 1}\label{sec:example1}
We let $T=0.1$, $\Omega$ be the unit cube, $\Gamma^u_D = \partial \Omega$, and $\Gamma^\varphi_D = \partial \Omega \setminus \{x_3=0 \text{ or } x_3=1\}$. To construct an example where the exact solution is known, we consider non-zero Neumann data $g_N$ for $\varphi$ and an additional function $f$ in the right-hand side of \eqref{weaku}. We get
\begin{align*}
\langle D_t u, v \rangle + \langle \nabla u, \nabla v \rangle &= -\langle \sigma(u)\lceil \tilde \varphi \rceil \nabla \varphi, \nabla v \rangle  +  \langle \sigma(u) \nabla \varphi \cdot \nabla g_\varphi, v\rangle + \langle f, v \rangle, \\
\langle \sigma(u) \nabla \varphi, \nabla w \rangle &= \langle g_N, w \rangle_{\Gamma^\varphi_N},
\end{align*} 
where $\langle \cdot, \cdot \rangle_{\Gamma^\varphi_N}$ denotes integration on the boundary $\Gamma^\varphi_N$.

Letting $g_u = t, g_\varphi = x_2, g_N = -1 + 2x_2$, $\sigma = 1$, and
\begin{align*}
f &= 2(x_1x_2(1-x_1)(1-x_2) + x_1x_3(1-x_1)(1-x_3) + x_2x_3(1-x_2)(1-x_3)),\\
u_0 &= x_1(1-x_1)x_2(1-x_2)x_3(1-x_3),
\end{align*}
the exact solution is given by $u = x_1(1-x_1)x_2(1-x_2)x_3(1-x_3) + t$ and $\varphi = x_2$. Note that $\varphi = \tilde \varphi + g_\varphi$ and thus $\tilde \varphi = 0$. In our setting, the approximations $\varphi_{m,l}$ are all close to zero and, hence, we omit to plot the error for $\varphi$ below.

We compute the finite element approximation on meshes with tetrahedra of maximal diameter $h=2^{-k}\sqrt{3}$ and time step size $\tau=2^{-k}T=2^{-k}0.1$ for $k=1,..,6$. With this refinement, the finest approximation ($k=6$) is computed on a mesh with 274625 nodes. The error in $L_2(0,T;H^1(\Omega))$ is approximated using Simpson's rule in time on each interval $I_n$ and the FEniCS function \verb|errornorm| in space. The relative error is depicted in Figure~\ref{fig:cube}. The convergence rate is approximately linear, which is expected for sufficiently regular problems.

\begin{figure}[h]
	\centering
	\includegraphics[width=0.7\textwidth]{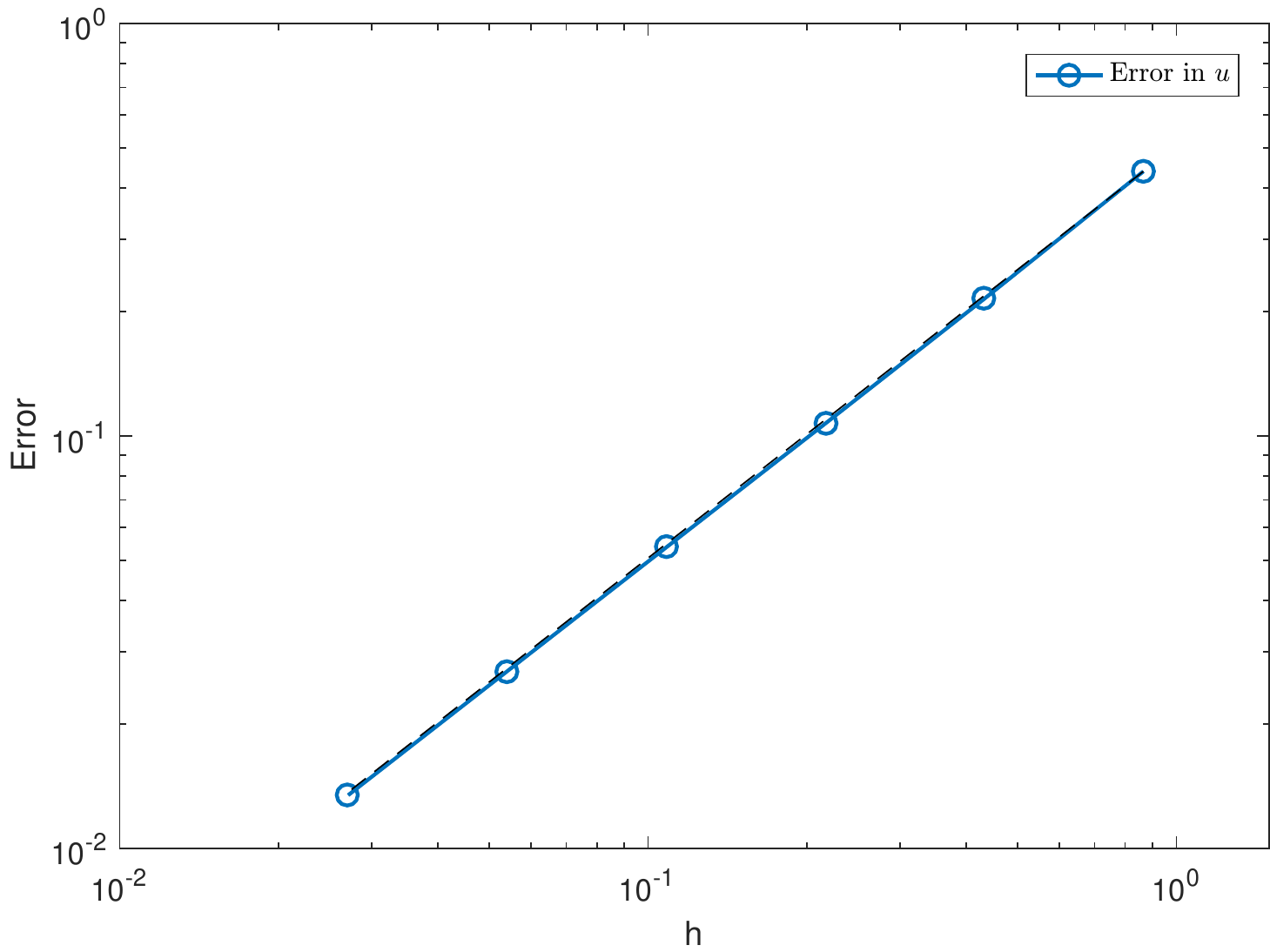}
	\caption{Relative errors for the temperature $u$ (blue o) of Example 1 plotted against the mesh size $h$. The dashed line is $Ch$.}\label{fig:cube}
\end{figure}

\subsection{Example 2}\label{sec:example2}
We let $T=0.1$ and $\Omega$ be the Fichera cube depicted in Figure~\ref{fig:fichera_domain} (left). We consider non-creased boundary conditions by imposing Dirichlet conditions on the striped areas, $\Gamma_0$ and $\Gamma_1$ in the figure (left), and homogeneous Neumann conditions on the remaining parts. On $\Gamma_0$ we set $g_u = 0$ and $g_\varphi = 10$ and on $\Gamma_1$ we set $g_u=g_\varphi=0$. Furthermore, we let $\sigma(u) = 2^{-1}(\pi - \arctan(u))$ and $u_0 = 0$. 

We compute the finite element approximation on meshes with tetrahedra of maximal diameter $h=2^{-(k-1)}\sqrt{2}$ and time step size $\tau=2^{-k}T=2^{-k}0.1$ for $k=1,..,5$. Since the exact solution is not known, the approximations are compared to a reference solution computed for $k=6$ corresponding to a mesh with 471233 nodes. The relative error in $L_2(0,T;H^1(\Omega))$-norm is plotted in Figure~\ref{fig:fichera}. We have convergence, but not with order one. This is due to the low regularity in the vicinity of the edges where the Dirichlet and Neumann boundaries meet with an angle greater than $\pi$, that is, where the creased domain condition fails. 

\begin{figure}[h]
	\centering
	\begin{subfigure}[b]{0.49\textwidth}
		\includegraphics[width=\textwidth]{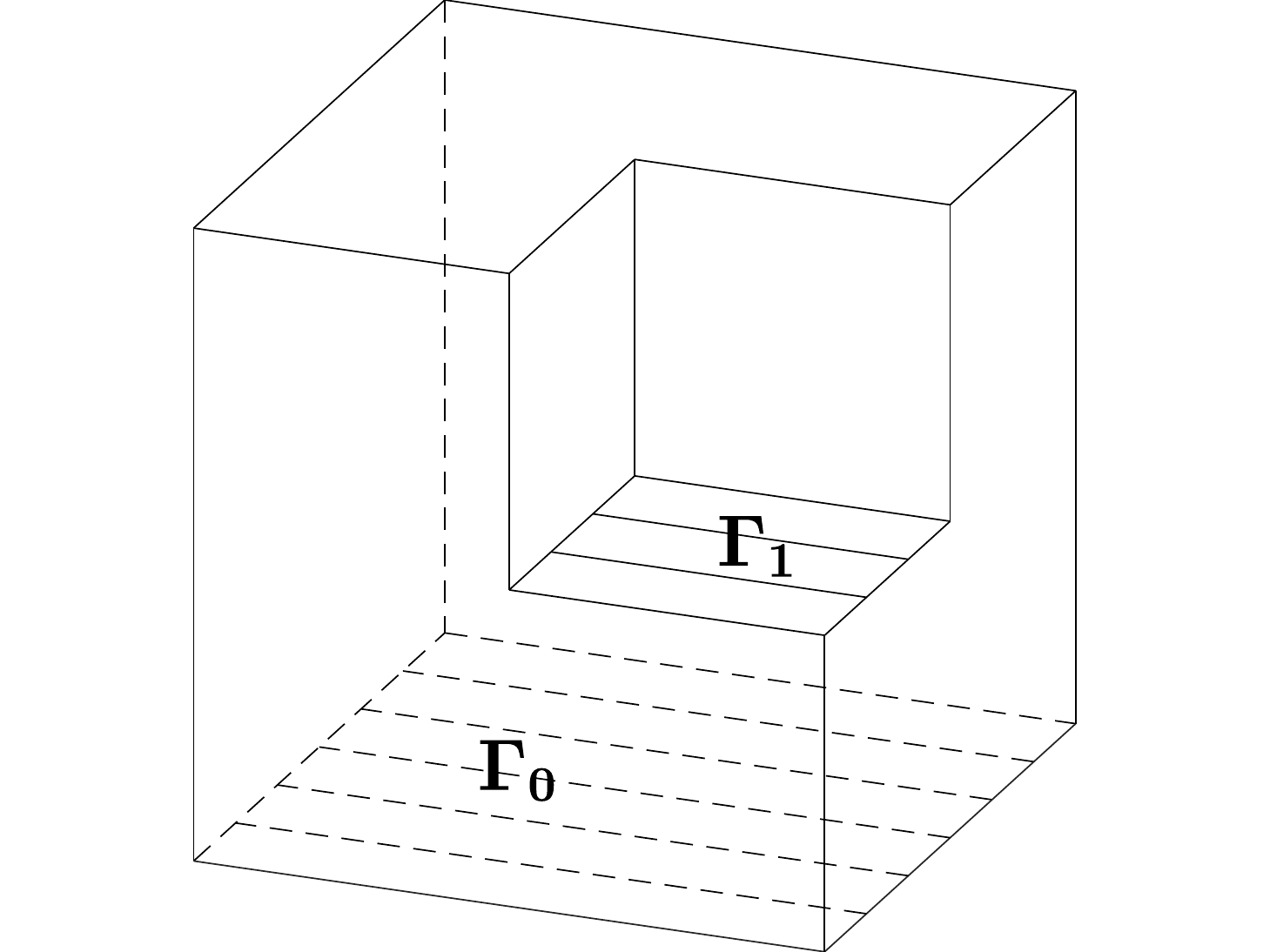}
		\caption{Non-creased}
	\end{subfigure}
	~ 
	\begin{subfigure}[b]{0.49\textwidth}
		\includegraphics[width=\textwidth]{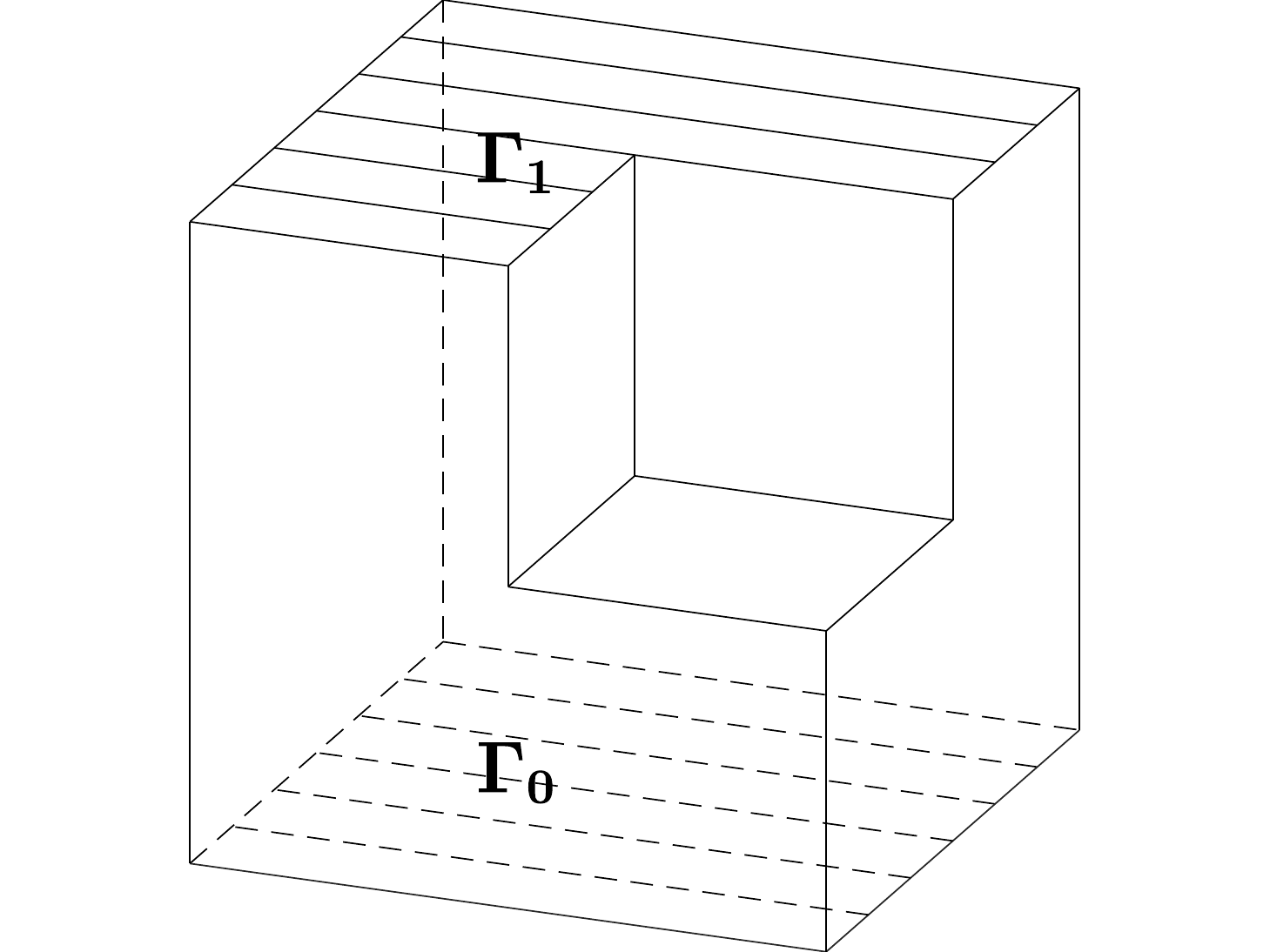}
		\caption{Creased}
	\end{subfigure}
	\caption{Two settings of the Fichera cube with centre in the origin. Dirichlet boundary conditions are imposed on the striped areas ($\Gamma_0$ and $\Gamma_1$).}\label{fig:fichera_domain}
\end{figure}
  
\begin{figure}[h]
	\centering
	\includegraphics[width=0.7\textwidth]{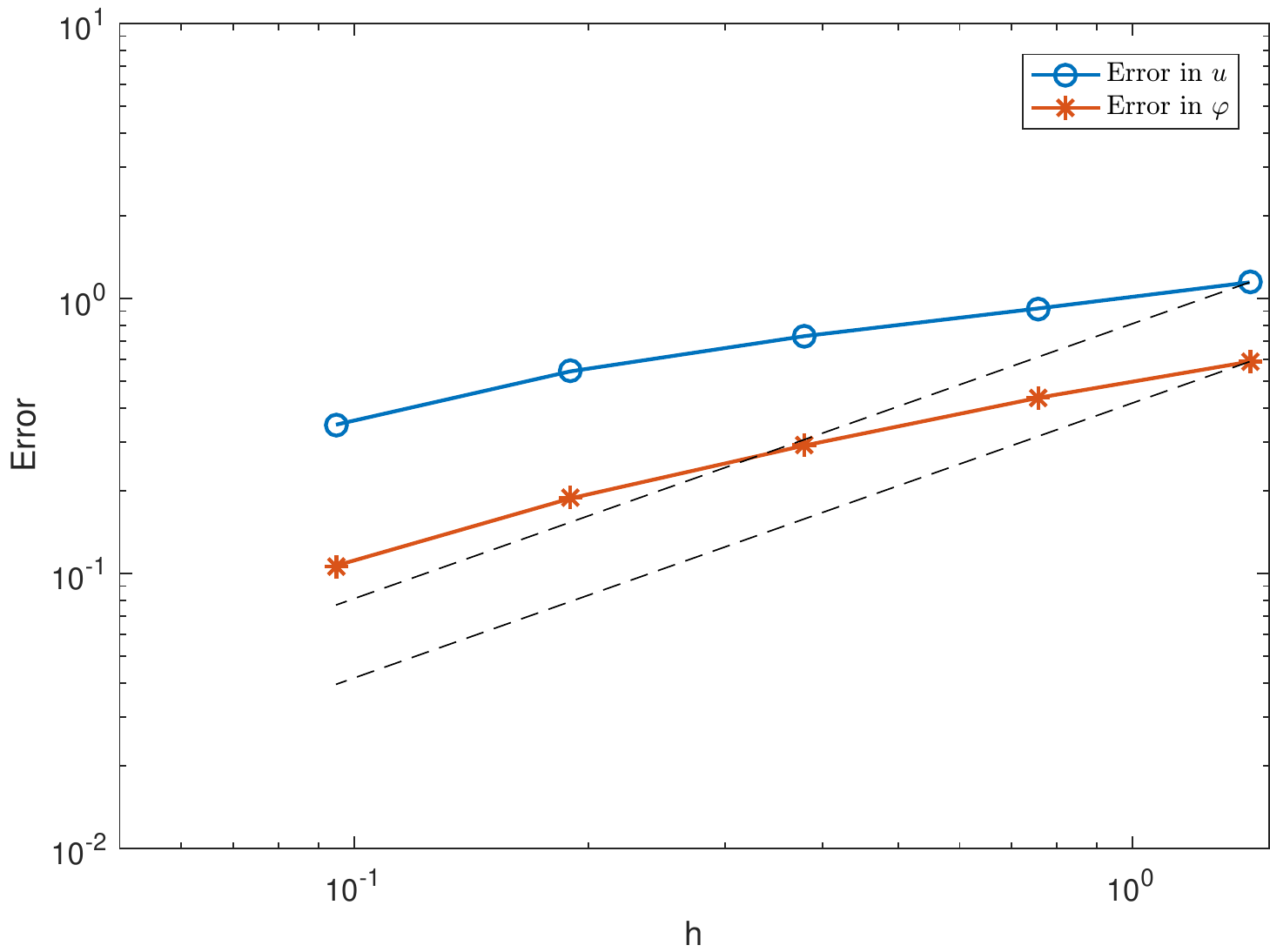}
	\caption{Relative errors for the temperature $u$ (blue $\circ$) and the potential $\varphi$ (red $\ast$) of Example 2 plotted against the mesh size $h$. The dashed line is $Ch$.}\label{fig:fichera}
\end{figure}

\subsection{Example 3}\label{sec:example3}
We continue in the setting of the Fichera cube as in Example~2, but with a choice of boundary conditions that fulfil the creased domain condition. We choose $\Gamma_0$ and $\Gamma_1$ as in Figure~\ref{fig:fichera_domain} (right), with $g_u = 0$ and $g_\varphi(x,t) = 2x_2(x_2+1)+5$ on both $\Gamma_0$ and $\Gamma_1$. 

The aim is to utilize the observation that the solution has higher regularity in the interior of the domain, see Theorem~\ref{thm_interior_reg}, and that the problem thus is suitable for $h$-adaptive finite elements. In this example we use a goal-oriented approach for the mesh refinement, which is supported for stationary problems in the FEniCS software, see \cite{Rognes2013}.

We summarize the goal-oriented procedure here, and refer to \cite{Rognes2013} and references therein for details. Consider a nonlinear variational problem; find $u \in V$ such that
\begin{align}\label{goal_stationary}
F(u,v) = 0, \quad \forall v \in \hat V,
\end{align}
and the corresponding finite element problem; find $u_h \in V_h$ such that
\begin{align}\label{goal_stationary_fem}
F(u_h,v) = 0, \quad \forall v \in \hat V_h,
\end{align}
for some triangulation $\mathcal T_h$ and appropriate finite element space $V_h \subset V$, $\hat V_h \subset \hat V$. Let $\M: V \to \R$ denote a linear goal functional and define the dual problem; find $z \in V^\ast$ such that 
\begin{align*}
\overline{F'}^\ast(z,v) = \M(v), \quad \forall v \in \hat V^\ast,
\end{align*}
where $\hat V^\ast = V_0 = \{v-w: v,w \in V\}$ and $V^\ast = \hat V$. The bilinear form $\overline{F'}^\ast$ denotes the following average of the Fr\'{e}chet  derivative $F'$ of $F$,
\begin{align*}
\overline{F'}(\cdot,\cdot) = \int^1_0 F'(su + (1-s)u_h;\cdot,\cdot) \intd s,
\end{align*}
and by the chain rule we have $\overline{F'}(u-u_h,\cdot) = F(u,\cdot) - F(u_h, \cdot)$. 

Using the definition of the dual problem we may now express the error in the goal functional as
\begin{align*}
\M(u)-\M(u_h) &= \M(u-u_h) = \overline{F'}^\ast(z,u-u_h) = \overline{F'}(u-u_h,z) \\&= F(u,z) - F(u_h, z) = - F(u_h, z) =:r(z),
\end{align*}
where $r(z)$ denotes the residual. The residual can be decomposed into local contributions from each cell $T \in \mathcal T_h$
\begin{align*}
r(v) = \sum_{T \in \mathcal T_h} r_T(v) = \sum_{T \in \mathcal T_h}\Bigg(\int_TR_Tv \intd x+ \int_{\partial T} R_{\partial T} v\intd s \Bigg),
\end{align*}
where $R_T$ and $R_{\partial T}$ are the cell and facet residuals. In \cite[Theorem 4.1]{Rognes2013} it is proved that the error indicators $R_T, R_{\partial T}$ can be determined by solving a set of local problems on each cell $T$. 

The procedure of computing the error indicators $R_T$ and $R_{\partial T}$ and refining the mesh accordingly is performed in FEniCS by using \verb|solve| together with the goal functional and a given tolerance. In our case, the fully discrete problem \eqref{full_weak} is a stationary problem of the form \eqref{goal_stationary} with 
\begin{align*}
F((u^n_{m,l},\varphi^n_{m,l}),(v,w)) &:= \langle\frac{ u^n_{m,l}-u^{n-1}_{m,l}}{\tau_l}, v \rangle + \langle \nabla u^n_{m,l}, \nabla v \rangle \\&\quad+\langle \sigma(u^n_{m,l})\lceil \tilde \varphi^n_{m,l} \rceil \nabla \varphi^n_{m,l}, \nabla v \rangle \\&\quad - \langle \sigma(u^n_{m,l}) \nabla \varphi^n_{m,l} \cdot \nabla g^n_\varphi, v\rangle + \langle \sigma(u^n_{m,l}) \nabla \varphi^n_{m,l}, \nabla w \rangle.
\end{align*} 
In each time step the error indicators are computed and the mesh refined. Note that the refined mesh is reused in the next time step and additionally refined if needed. 

In this example we choose $\M(u) = \int_{\Omega} u \intd x$. The initial data remains the same as in Example 2. We choose to have fixed (small) time step $\tau = 2^{-6} T$ in this experiment, since the the spatial error is the main concern here. The relative error in the goal functional for $h=2^{-4}\sqrt{2}$ compared to the reference solution, here denoted $u_\textrm{ref}$ and computed on mesh with 471233 nodes, is
\begin{align*}
 \frac{\displaystyle \max_{0 \leq n \leq N} \left| \M(u^n_h) - \M(u^n_\textrm{ref}) \right|}{\displaystyle \max_{0 \leq n \leq N}\left| \M(u^n_\textrm{ref}) \right|} \approx 0.0254. 
\end{align*} 
We note that our uniform mesh of size $h=2^{-4}\sqrt{2}$ corresponds to $7985$ nodes. Using the goal oriented adaptivity, denoted $u_\textrm{ad}$ below, we get  
\begin{align*}
 \frac{\displaystyle \max_{0 \leq n \leq N} \left| \M(u^n_\textrm{ad}) - \M(u^n_\textrm{ref}) \right|}{\displaystyle \max_{0 \leq n \leq N}\left| \M(u^n_\textrm{ref}) \right|} \approx 0.0282, 
\end{align*}
already for $1628$ nodes. 

This example indicates that the problem is suitable for $h$-adaptive finite elements and motivates a further analysis of a posteriori methods for the Joule heating problem, which will be considered in later works.

\subsection{Example 4}\label{sec:example4}
In this example, we use the non-creased Fichera cube as in Example~2, see Figure~\ref{fig:fichera_domain} (left). The aim is to investigate the use of goal-oriented adaptivity for non-creased domains. We emphasize that, in this setting, Theorem~\ref{thm_interior_reg} is not directly applicable.

As in Example~3 we choose $\M(u) = \int_{\Omega} u \intd x$. The initial and boundary data remain the same as in Example 2 and the time step is $\tau = 2^{-6} T$. The error in the goal functional for $h=2^{-5}\sqrt{2}$ compared to the reference solution is
\begin{align*}
\frac{\displaystyle \max_{0 \leq n \leq N} \left| \M(u^n_h) - \M(u^n_\textrm{ref}) \right|}{\displaystyle \max_{0 \leq n \leq N}\left| \M(u^n_\textrm{ref}) \right|} \approx 0.0271. 
\end{align*}  
Here $h=2^{-5}\sqrt{2}$ corresponds to $60513$ nodes. For the goal oriented adaptivity we get  
\begin{align*}
\frac{\displaystyle \max_{0 \leq n \leq N} \left| \M(u^n_\textrm{ad}) - \M(u^n_\textrm{ref}) \right|}{\displaystyle \max_{0 \leq n \leq N}\left| \M(u^n_\textrm{ref}) \right|} \approx 0.0254, 
\end{align*}
for $6560$ nodes.


This example indicates that the goal oriented adaptivity is applicable also in non-creased domain settings. However, it is still an open problem to show that the solution to such a problem enjoys the appropriate regularity to be suitable for $h$-adaptivity.



\subsection*{Acknowledgement}
The authors acknowledge the hospitality of the Hausdorff Research Institute for Mathematics in Bonn, where parts of this paper were written.

\bibliographystyle{abbrv}
\bibliography{jouleheating_ref}

\end{document}